\documentclass{amsart}

\usepackage{amsmath, amsthm, amssymb, amsfonts, mathtools}
\usepackage{graphicx,epsfig,color}
\usepackage{scrpage2}
\usepackage{exscale}

\theoremstyle{plain}

\DeclareMathOperator{\Cpl}{Cpl}

\usepackage{comment}

\newtheorem{theorem}{Theorem}[section]

\newtheorem{lemma}[theorem]{Lemma}

\newtheorem{corollary}[theorem]{Corollary}

\newtheorem{question}[theorem]{Question}
\newtheorem{proposition}[theorem]{Proposition}

\theoremstyle{definition}

\newtheorem{definition}[theorem]{Definition}

\theoremstyle{remark}
\newtheorem{remark}[theorem]{Remark}

\newtheorem{example}[theorem]{Example}

\newcommand{\VK}{\kappa}
\newcommand{\sk}{\kappa}

\DeclareSymbolFont{AMSb}{U}{msb}{m}{n}
\DeclareMathSymbol{\N}{\mathalpha}{AMSb}{"4E}
\DeclareMathSymbol{\R}{\mathalpha}{AMSb}{"52}
\DeclareMathSymbol{\Z}{\mathalpha}{AMSb}{"5A}
\DeclareMathSymbol{\D}{\mathalpha}{AMSb}{"44}
\DeclareMathSymbol{\s}{\mathalpha}{AMSb}{"53}

\DeclareMathOperator{\sph}{s}

\DeclareMathOperator{\essinf}{essinf}

\newcommand{\sinfty}{{\scriptscriptstyle{\infty}}}
\newcommand{\sR}{\scriptscriptstyle{R}}

\newcommand{\sX}{\scriptscriptstyle{X}}
\newcommand{\sM}{\scriptscriptstyle{M}}
\newcommand{\sW}{\scriptscriptstyle{W}}
\newcommand{\sN}{\scriptscriptstyle{N}}
\newcommand{\sK}{\scriptscriptstyle{K}}
\newcommand{\sZ}{\scriptscriptstyle{Z}}
\newcommand{\sPi}{\scriptscriptstyle{\Pi}}
\newcommand{\sY}{\scriptscriptstyle{Y}}
\newcommand{\sscr}{\scriptscriptstyle}
\newcommand{\mms}{X}
\newcommand{\mmsc}{[X
]}
\newcommand{\mmsi}{(X_i)}
\newcommand{\mmsic}{[X_i]}
\newcommand{\pmms}{(X, o)}
\newcommand{\pmmsc}{[X
,o]}
\newcommand{\pmmsi}{(X_i,o_i)}
\newcommand{\pmmsic}{[X_i,o_i]}
\newcommand{\pmmsinf}{(X_{\sinfty},o_{\sinfty})}

\newcommand{\pmmscinf}{[X_{\sinfty},o_{\sinfty}]}

\newcommand{\mmsinfc}{[X_{\sinfty}]}

\newcommand{\dex}{\de_{\sX}}

\DeclareMathOperator{\frs}{\sin}
\DeclareMathOperator{\frc}{\cos}
\DeclareMathOperator{\vol}{vol}
\DeclareMathOperator{\Ch}{Ch}
\DeclareMathOperator{\lip}{Lip}

\DeclareMathOperator{\supp}{supp}

\DeclareMathOperator{\de}{d}
\DeclareMathOperator{\m}{m}
\DeclareMathOperator{\ric}{ric}
\DeclareMathOperator{\diam}{diam}

\DeclareMathOperator{\Ent}{Ent}

\newcommand{\ChX}{\Ch^{\sX}}

\title{Stability of metric measure spaces with  integral Ricci curvature bounds}

\author{Christian Ketterer}
\address{Department of Mathematics, 
University of Toronto,  
40 St George St,
Toronto Ontario,  
Canada M5S 2E4}
\thanks{
CK is funded by the Deutsche Forschungsgemeinschaft (DFG, German Research Foundation) -- Projektnummer 396662902, ''Synthetische Kr\"ummungsschranken durch Methoden des optimal Transports''.}
\email{ckettere@math.toronto.edu.}

\begin{document}
\maketitle
\begin{abstract}
In this article we study  stability and compactness w.r.t.  measured Gromov-Hausdorff convergence of smooth metric measure spaces with integral Ricci curvature bounds.
More precisely, we  prove that a sequence of  $n$-dimensional Riemannian manifolds subconverges to a metric measure space that satisfies the curvature-dimension condition $CD(K,n)$ in the sense of Lott-Sturm-Villani provided the $L^p$-norm for $p>\frac{n}{2}$ of the part of the Ricci curvature that lies below $K$ converges to $0$. The results also hold  for sequences of general smooth metric measure spaces $(M,g_M, e^{-f}\vol_M)$ where  Bakry-Emery curvature replaces Ricci curvature. Corollaries are a  Brunn-Minkowski-type inequality, a Bonnet-Myers estimate and a statement on finiteness of the fundamental group. Together with a uniform noncollapsing condition the limit even satisfies the Riemannian curvature-dimension condition $RCD(K,N)$. This implies  volume and diameter almost rigidity theorems.
\end{abstract}
\tableofcontents
\section{Introduction}
Stability and compactness properties  of families of  Riemannian manifolds satisfying a uniform estimate on the $L^p$-norm of the part of their Ricci curvature that lies below a given  treshold $K\in \R$  have been a topic of increasing  interest in recent years, e.g. \cite{petersenwei, aubry}.
The crucial quanitity   for a compact  Riemannian manifold $M$ is the \textit{integral curvature excess}
\begin{align*}
\left(\frac{(\diam M)^{2p}}{\vol_M(M)}\int (\kappa-K)_{-}^p d\vol_M\right)^{\frac{1}{p}}=: k_{[M]}(\kappa,p,K)
\end{align*}
where $(\kappa-K)_-=-\min\{\kappa-K,0\}$, $\ric_M\geq \kappa $ for $\kappa\in C(M)$, $p>\frac{N}{2}$ and $ \dim_M\leq N<\infty$. 

There have been numerous publications that study Riemannian manifolds with  bounded integral curvature excess, e.g. \cite{gallotintegral, yang, yang2, petersenwei, petersenwei2,  psw, dpw, daiwei, aubry, rose2, dwzsobolev, rosewei}.  Remarkable properties  are pre-compactness under Gromov-Hausdorff convergence, effective diameter and volume growth estimates and  estimates on the spectral gap. These results  resemble corresponding statements for lower Ricci curvature bounds and typically involve error terms that  depend  on  the integral curvature excess provided it is sufficiently small. One can  construct examples by gluing  small cusps of arbitrary negative curvature on Riemannian manifold with  a lower Ricci curvature bound. These cusps can be constructed such that the curvature excess is arbitrarily small and a subsequence will converge in Gromov-Hausdorff sense to the original manifold.

In general the precompactness property suggests that there is a theory of nonsmooth limit spaces with integral curvature bounds. Moreover, the effective estimates indicate some stability property under measured Gromov-Hausdorff convergence.  
One  conjectures that the measured Gromov-Hausdorff limit of a sequence of Riemannian manifolds with vanishing integral curvature excess $k_{[M]}(\kappa,p,K)$  is a nonsmooth metric measure space with Ricci curvature bounded from below by $K$. A nonsmooth theory of lower Ricci curvature bounds  is in fact provided by the class of metric measure spaces that satisfy a curvature-dimension condition $CD(K,N)$ for $K\in \mathbb{R}$ and $N\geq 1$. This condition was   introduced in  celebrated  works  by Lott, Sturm and Villani \cite{stugeo1, stugeo2, lottvillani}.

In this article we  confirm the previous conjecture. Our main result reads as follows.
\begin{theorem}\label{th:riemcase}Let $D\geq 0$, $N\geq 2$ and $K\in \mathbb{R}$. 

For every $\epsilon>0$ there exists $\delta>0$ such that the following holds. Let
$(M,g_M)$ be a Riemannian manifold
such that $\ric_{\sM}\geq \kappa$ for $\kappa\in C(M)$, $\diam M\leq D$, $\dim M\leq N$ and
$$
{k}_{[M, d_{\sM},\vol_{\sM}]}(\kappa, p,K)<\delta\ \mbox{ 
for $p>\frac{N}{2}$ and $p\geq 1$ if $N=2$.}
$$
Then there exists a metric measure space $\mms$ that satisfies the curvature-dimension condition $CD(K,N)$ and 
\begin{align*}
\mathbb{D}([M,\de_M,\vol_M], [\mms])<\epsilon. 
\end{align*}
\end{theorem}
Here $\mathbb D$ is  the so-called {\it transport distance} introduced by Sturm in \cite{stugeo1}. It is a distance on  isomorphism classes of metric measure spaces and metrizes the notion of {\it measured Gromov convergence} that was introduced in \cite{gmsstability}. 
The isomorphism class of a metric measure space $X$ is denoted with $[X]$.
The theorem does not follow from the  compactness property of a family of metric measure spaces that satisfy a uniform lower Ricci curvature bound
since bounds on the integral curvature excess don't imply uniform lower Ricci curvature bounds in general.

Now let us recall  that a bound of the form $\ric_M\geq \kappa$ for $\kappa\in C(M)$ together with $N\geq \dim_M$ can be characterized in terms of a  curvature-dimension condition $CD(\kappa,N)$ that was  introduced by the author in \cite{ketterer5}. For smooth metric measure spaces $(M,d_M,\m_M)$, that is a Riemannian manifold  equipped with a smooth measure $e^{-f}\vol_M=:\m_M$ such that $f\in C^\sinfty(M)$, the same characterization holds with $\ric_M$ replaced by the so-called $N$-Bakry-Emery curvature. In \cite{ketterer5} it was observed that the condition $CD(\kappa,N)$ for a continuous function $\kappa$  makes sense for any metric measure space $X$ and generalizes the theory of Lott-Sturm-Villani to a setup of variable lower Ricci curvature bounds.

Then Theorem \ref{th:riemcase}  is  a special case of the following theorem  for pointed, smooth, metric measure spaces $(M,o)$ where $o$ is a fixed base point in $M$. For this setup we consider the integral curvature excess centered at $o$ with radius $R>0$:
\begin{align*}
\left(\frac{R^{2p}}{\m_M(B_1(o))}\int_{B_R(o)} (\kappa-K)_{-}^p d\m_M\right)^{\frac{1}{p}}=: k_{[M,o]}(\kappa,p,K,R).
\end{align*}
We say a pointed metric measure space $(X,o)$ is normalized if $\m_{\sX}(B_1(o))=1$. We  prove the following result.
\begin{theorem}\label{th:mainmain}Let $N\in [2,\infty)$ and $K\in \mathbb{R}$. 

Let $\{(M_i,o_i)\}_{i\in\mathbb{N}}$ be a sequence of smooth, normalized, pointed metric measure spaces that satisfy a condition $CD(\kappa_i,N)$ for $\kappa_i\in C(X_i)$
and
$$
k_{[M_i,o_i]}(\kappa_i,p,K,R)\rightarrow 0 \mbox{ when  } i\rightarrow \infty \ \ \forall R>0, p>\frac{N}{2} \mbox{ and }p\geq 1 \mbox{ if }N=2.
$$
Then the isomorphism classes $\{[M_i,o_i]\}_{i\in\mathbb{N}}$ subconverge in pointed measured Gromov sense  to 
the isomorphism class of a pointed, normalized metric measure space $\pmms$ that satisfies the condition $CD(K,N)$.
\end{theorem}
The assumption $p>\frac{N}{2}$ for $N>2$ is sharp since Aubry showed in \cite{aubry} that compact $N$-dimensional Riemannian manifolds with $k_{[M]}(\kappa, \frac{N}{2}, K)<\epsilon$ are dense w.r.t.  Gromov-Hausdorff convergence among all compact length spaces.

Instead of requiring that the part of the Ricci curvature below $K$ is in some $L^p$-space for $p>\frac{N}{2}$, one can also assume that it satisfies a  {\it Kato condition}. This condition is strictly weaker than the previous $L^p$-condition but  Riemannian manifolds satisfying such a condition still have  properties that resemble the ones under integral curvature bounds, e.g. \cite{rosestollmann, rosestollmannsurvey}. Hence,  an extention of our theorem in this direction seems to be tangible.

As part of the proof of Theorem \ref{th:mainmain} we derive a displacement convexity inequality (Theorem \ref{resultA}) that implies the following, new Brunn-Minkowski-type inequality under integral curvature bounds.

\begin{corollary}[Brunn-Minkowski inequality]\label{cor:bm}
Let $M$ be a smooth, normalized mm space that satisfies $CD(\kappa,N)$ for $\kappa\in C(M)$ and $N\geq 2$. 
Let $p>{\textstyle \frac{N}{2}}$ such that $k_{[\sM]}(p,0)<\infty$ and 
let $A_0,A_1\subset M$. 
Then there exists a positive constant $C=C(p,N,K)$ (see Remark \ref{rem:const}) such that
\begin{align}
\m_{\sM}(A_t)^{\frac{1}{N}}&\geq (1-t) \m_{\sM}(A_0)^{\frac{1}{N}}+ t\m_{\sM}(A_1)^{\frac{1}{N}}\nonumber\\
&\hspace{1.5cm}-2C^{\frac{1}{N(2p-1)}} k_{[M]}(\kappa,p,0)^{\frac{p}{N(2p-1)}} \ \forall t\in (0,1)\nonumber
\end{align}
where $A_t=\{\gamma(t): \gamma\in \mathcal G(M), \gamma(0)\in A_0, \gamma(1)\in A_1\}$.
\end{corollary}
Moreover we prove the following precompactness result for nonsmooth and  non-branching metric measure spaces that was  proved for Riemannian manifolds in \cite{petersenwei}.
\begin{theorem}
Consider the family $\mathcal{X}(p,K,N)$ of isomorphism classes of essentially nonbranching pointed metric measure spaces $(X,o)$ satisfying a condition 
\linebreak[4] $CD(\kappa,N)$  for some $\kappa\in C(X)$
such that ${k}_{[X,o]}(\kappa, p,K,D)\leq f(D)$ for all $[X.o]\in \mathcal X(p,K,N)$,  for all $D\geq 1$, $p>\frac{N}{2}$ and some function $f:[1,\infty)\rightarrow \mathbb R$.
Then $\mathcal{X}(p,K,N)$ is precompact in the sense of pointe measured Gromov convergence.
\end{theorem}
\begin{remark}
In this respect we also mention recent work by Sturm \cite{sturmmv} where a class of metric measure spaces with lower curvature bounds in distributional sense is introduced. 
\end{remark}

In \cite{tianzhang} Tian and Zhang develop a regularity theory for limits of $n$-dimensional Riemannian manifolds $M$ such that 
\begin{align}\label{uniform2}
\mbox{$\ric_M\geq \kappa$, $k_{[M]}(\kappa, p, 0)\leq \Lambda <\infty$ with $p>\frac{n}{2}$. }
\end{align}They introduce the following
 uniform (and infinitesimal) noncollapsing condition: There exists $\varkappa>0$ such that 
\begin{align}\label{noncollapsing}
\vol_{\sM_i}(B_r(x))\geq \varkappa r^n \ \ \forall x\in M_i, \forall r\in (0,1) \mbox{ and }\forall i\in \N. 
\end{align}
This property is then used to develop a Cheeger-Colding-Naber type theory for limits of a sequence of manifolds satisfying \eqref{uniform2} and \eqref{noncollapsing}. An example by Yang \cite{yang} shows that a noncollapsing condition on a definite scale is not sufficient and the splitting theorem fails.  Assuming the uniform noncollapsing condition \eqref{noncollapsing} Tian and Zhang develop a satisfying regularity theory for limits that arise from manifolds satisfying $k_{[M]}(\kappa,p,0)\leq \Lambda <\infty$. In particular they prove  an almost splitting theorem and an almost-volume-cone-implies-almost-metric-cone theorem.  One can apply these results to our situation. The almost splitting theorem carries over to a splitting theorem of the limit space in Theorem \ref{th:mainmain}. We therefore obtain the following theorem.
\begin{theorem}\label{th:noncol}
Let $\{(M_i,o_i)\}_{i\in\mathbb{N}}$ be a sequence of $n$-dimensional, pointed Riemannian manifolds that satisfy the condition $CD(\kappa_i,N)$ for $\kappa_i\in C(X_i)$ 
such that 
$$
k_{[M_i,o_i]}(\kappa_i, p,K,R)\rightarrow 0 \mbox{ when  } i\rightarrow \infty \ \ \forall R>0
$$
with $K\in \R$,  $p>\textstyle{\frac{N}{2}}$, $p\geq 1$ if $N=2$ and \eqref{noncollapsing} holds.
Then $\{[M_i,o_i]\}_{i\in\mathbb{N}}$ subconverges in pointed measured Gromov sense to 
the isomorphism class of a pointed metric measure space $\pmms$ satisfying the Riemannian curvature-dimension condition $RCD(K,N)$.
\end{theorem}
\begin{corollary}
Let $X$ be a measured Gromov-Hausdorff limit of a sequence of Riemannian manifolds satisfying \eqref{uniform2} and \eqref{noncollapsing}. Then every tangent space $T_xX$ for $x\in X$ satisfies the condition $RCD(0,N)$. In particular, $T_xX$ is an Euclidean cone over some $RCD(N-2,N-1)$ space $Y$.
\end{corollary}
We note that the example in \cite{yang}  does not satisfy $k_{[M]}(p,K)\rightarrow 0$. Therefore we  expect that  our main theorem can be improved. 
We raise the following question.
\begin{question}
In Theorem \ref{th:riemcase} can we replace the condition $CD(K,N)$ with  the Riemannian condition $RCD(K,N)$.
\end{question}
A theorem for $RCD(K,2)$ spaces by Lytchak and Stadler \cite{lyst} also yields the following corollary.
\begin{corollary}
Let $\{(M_i,o_i)\}_{i\in\mathbb{N}}$ be a sequence of $2$-dimensional, pointed Riemannian manifolds that satisfy the condition $CD(\kappa_i,2)$ for $\kappa_i\in C(X_i)$ 
such that 
$$
k_{[M_i,o_i]}(\kappa_i,1,K,R)\rightarrow 0 \mbox{ when  } i\rightarrow \infty \ \ \forall R>0
$$
with $K\in \R$ and \eqref{noncollapsing} holds.
Then $\{[M_i,o_i]\}_{i\in\mathbb{N}}$ subconverges in pointed Gromov-Hausdorff sense to a pointed $2$-dimensional Alexandrov space with curvature bounded from below by $K$.
\end{corollary}
Let us briefly explain  the main ideas in the proof of Theorem \ref{th:mainmain}.
Assume for simplicity $K=0$.
At the core of our proof is a new displacement convexity inequality for the $N$-Renyi entropy functional (Theorem \ref{resultA}). This inequality is similar to  corresponding inequalities under lower Ricci curvature bounds but involves an error term that explicitly depends on the integral curvature excess. The proof of this result consists of three steps. First, we analyse carefully  the $1$-dimensional model case. This allows us to prove estimates  for the {\it modified distortion coefficiants} $\tau_{\kappa,N}^{(t)}(\theta)$ involving $L^p$ integrals of $\kappa$. Here, $\kappa$ is a continuous function $[0,\theta]$. The coefficients $\tau_{\kappa,N}^{(t)}(\theta)$ (before Definition \ref{bigg}) play a crucial role in the definition of the condition $CD(\kappa, N)$ for a metric measure space $X$ and a variable lower curvature  bound $\kappa: X\rightarrow \mathbb R$. Second, in Section 5 we will apply the Area and Co-area formula to a  transport Kantorovich potential and derive two disintegrations of the reference measure $e^{-V}\vol_M$ (Proposition \ref{prop:areaformula} and Lemma \ref{lem:coarea}). Proposition \ref{prop:areaformula}  resembles a similar disintegration obtained by $L^1$ optimal transport (see \cite{cavmon} and in particular \cite{cavallettimilman}). However, since we are interested in sequences of smooth spaces,  we can use classical tools of geometric analysis that are sufficient for this setup.  Finaly, starting from the localized version of the $CD(\kappa,N)$ condition we put together the previous steps to obtain the desired displacement-convexity-type inequality (Section \ref{section:displacement}).
\subsection{Applications}
Here we present some immediate consequences that derive from the main theorems and its corollaries.  First, we can prove a Bonnet-Myers diameter type bound that improves and generalizes a result by Aubry \cite{aubry}.
\begin{corollary}
Let $\{(M_i,o_i)\}_{i\in\mathbb{N}}$ be a sequence of smooth, normalized pmm spaces that satisfy the condition $CD(\kappa_i,N)$ for $\kappa_i\in C(X_i)$ 
such that
$$
k_{[M_i,o_i]}(\kappa_i, p,K,R)\rightarrow 0 \mbox{ as } i\rightarrow \infty \ \ \forall R>0
$$
with $K>0$ and $p>\textstyle{\frac{N}{2}}$. For every $\epsilon>0$ there exists $i_\epsilon\in \N$ such that $M_i$ is compact for every $i\geq i_\epsilon$ and $\diam M_i\leq \pi_{K/(N-1)}+\epsilon$.
\end{corollary}
\noindent 
The proof is straightforward by arguing by contradiction.

A Bonnet-Myers type estimate for a Riemannian manifold satisfying a Kato condition was obtained in \cite{rosemyers, carronrose}.

As a consequence from the previous corollary we also obtain the following statement on finiteness of the  fundamental group.
\begin{corollary}
There exists a function $f:(0,\infty)\rightarrow (0,\infty)$ such that  the following holds.
If $(M,o)$ is  a  smooth metric measure space that satisfies the condition $CD(\kappa,N)$ for $\kappa\in C(M)$ 
and
$$
k_{[M,o]}(\kappa, p, R,K)\leq f(R) \mbox{ for all }R>0
$$
with $K>0$ and $p>\textstyle{\frac{N}{2}}$, then $M$ has finite fundamental group.
\end{corollary}
\begin{proof}Assume the statment fails. Then, there exists a sequence of smooth pointed metric measure spaces $(M_i,o_i)$ satisfying $CD(\kappa_i,N)$ such that $k_{[M_i]}(\kappa_i,p, R,K)\rightarrow 0$ for $i\rightarrow \infty$ and $\forall R>0$ and such that the fundamental group is not finite. 

First, by the previous Corollary   we can assume that $M_i$ is compact $\forall i\in \N$ and we can normalize the metric measure space $M_i$. Then we still have $k_{[M_i]}(\kappa_i,p,K)\rightarrow 0$.
Let $\tilde M_i$ be the Riemannian  universal cover of $M_i$ equipped with the pull back measure $\m_{\tilde M_i}$ under the convering map $p_i:\tilde M_i\rightarrow M_i$.
For every $i\in \N$ we choose a base point $o_i\in \tilde M_i$. Since the fundamental group is not finite, it follows $\diam \tilde M_i =\infty$.
Moreover it follows that the metric measure space $\tilde M_i$ satisfies the condition $CD(\kappa_i\circ p_i, N)$ and 
$$
k_{[\tilde M_i,o_i]}(\kappa_i \circ p_i, p,K,R)\rightarrow 0 \mbox{ as } i\rightarrow \infty \ \ \forall R>0.
$$
By the previous corollary there exists $i_{\epsilon}\in \N$ such that $\diam_{\tilde M_i}\leq \pi_{K/(N-1)}+\epsilon$ for $i\geq i_{\epsilon}$. 
Hence, $\tilde M_i$ is compact. That is a contradiction.
\end{proof}
A theorem on finiteness of the fundamental group of Riemannian manifold under an integral curvature condition appears in \cite{aubry}  and under a Kato conditon in \cite{carronrose, rosemyers}. Let us point out that our theorem  improves Aubry's result even in the Riemannian case since we do not require a priori that $(\kappa-K)_-$ is $L^p(\vol_M)$ integrable for $p>\frac{n}{2}$. A result for weighted graphs satisfying the Kato conditon appears in \cite{mr}.

\subsubsection{Almost rigidity results}
Recall the following definitions.
For $K>0$ and $N> 1$ the $1$-dimensional model space is $$I_{K,N}=\left(\left[0,{\pi_{K/(N-1)}}\right], 1_{\left[0,\scriptstyle{\pi_{K/(N-1)}}\right]}\sin_{K/(N-1)}^{N-1}\mathcal L^1\right)$$
where $[0,\scriptstyle{\pi_{K/(N-1)}}]$ is equipped with the restriction of the standard metric $|\cdot|$ on $\R$. The metric measure space
$I_{K,N}$ satisfies $CD(K,N)$ \cite[Example 1.8]{stugeo2}. 

Let $(M,g,\m)=M$ be a weighted Riemannian manifold with $\m={\Phi}\vol_g$ and $\Phi\in C^{\infty}(M\backslash \partial M)$. The warped product
$I_{K,N}\times_f^{N-1}M$ between $I_{K,N}$ and $M$ w.r.t. $f:I_{K,N}\rightarrow [0,\infty)$ is defined as the metric completion of the weighted Riemannian manifold
$\left(I_{K,N}\times M, h, \m_C\right)$
where $h=\langle\cdot,\cdot\rangle^2+ f^2 g$ and $\m_C=f^{N-1} \mathcal L^1|_{I_{K,N}} \otimes \m$. In \cite{ketterer} it was proved that if the warping function $f$ satisfies 
\begin{align*}
f''+ \frac{K}{N-1}f\leq 0\ \mbox{ and }\  (f')^2+\frac{K}{N-1}f^2\leq L \mbox{ on } I_{K,N}
\end{align*}
and $(M,d_g,\m)$ satisfies $CD(L(N-2),N-1)$ then $I_{K,N}\times^{N-1}_{f} M$ satisfies $CD(K,N)$. This applies in particular when $f=\sin_{K/(N-1)}$ and $L=1$. Then the corresponding warped product is a spherical suspension. For instance, we can choose $M=I_{N-2,N-1}$. If $n\in \N$ we can choose 
$M=\mathbb{S}^{n-1}_{1}$ and we get that
$
I_{K,n}\times^{n-1}_{\sin_{K/(n-1)}}\mathbb S^{n-1}_1=\mathbb{S}_{K/(n-1)}^n.
$

More generally, one can define warped products in the context of metric measure spaces. In \cite{ketterer2} it was proved that 
$
I_{K,N-1}\times_{\sin_{{K}/{N-1}}}^{N-1} Y
$ 
satisfies the condition $RCD(K,N)$ if and only if $Y$ satisfies the condition $RCD(N-2,N-1)$.

Rigidity statements for $RCD$ spaces together with Theorem \ref{th:noncol} yield the following almost rigidity statements for smooth metric measure spaces with integral curvature bounds.
\begin{corollary}
For every $\epsilon>0$ there exists $\delta>0$ such that the following holds. If $M^n$ is a compact  Riemannian manifold that satisfies \eqref{noncollapsing},  $k_{[M]}(p,n-1)<\delta$ for $p>\frac{n}{2}$ and $\diam M\geq \pi-\delta$, then there exists an $RCD(n-2,n-1)$ space $Y$ such that 
\begin{align*}
\mathbb{D}( [M,d_{\sM}, \vol_{\sM}], [0,\pi]\times^n_{\sin}Y)\leq \epsilon.
\end{align*}
\end{corollary}
\begin{corollary}
For every $\epsilon>0$ there exists $\delta>0$ such that the following holds. If $(M^n,o)$ is a pointed  Riemannian manifold that satisfies \eqref{noncollapsing},  $k_{[M]}(p,0)<\delta$ for $p>\frac{n}{2}$ and
\begin{align*}
\frac{\vol_M(B_{2r}(o))}{\vol_{\R^n}(B_{2r}(0))}\geq  (1-\delta)
\frac{\vol_M(B_r(o))}{\vol_{\R^n}(B_r(0))}
\end{align*}
for some $r>0$, 
then there exists an $RCD(n-2,n-1)$ space $Y$ such that 
\begin{align*}
\mathbb{D}( [M,d_{\sM}, \vol_{\sM}], [0,r]\times^n_{t}Y)\leq \epsilon.
\end{align*}
\end{corollary}
The second  corollary appears with weaker assumptions also in \cite{tianzhang}. But our result yields in addition that the cross section is an $RCD$ space.

Instead of the noncollapsing condition one can also add an upper curvature bound. This will also force the limit to become $RCD$ by \cite{cdmeetscat, kkweak, kkk}
\begin{corollary}
Let $\{(M_i,o_i)\}_{i\in\mathbb{N}}$ be a sequence of $n$-dimensional, pointed Riemannian manifolds that satisfy the condition $CD(\kappa_i,N)$ for $\kappa_i\in C(X_i)$ 
such that 
$$
k_{[M_i,o_i]}(p,K,R)\rightarrow 0 \mbox{ when  } i\rightarrow \infty \ \ \forall R>0
$$
with $K\in \R$,  $p>\textstyle{\frac{N}{2}}$, $p\geq 1$ if $N=2$ and $M_i$ satisfies a $CAT(\bar{K})$ condition with  $\bar K\in \R$.
Then $\{[M_i,o_i]\}_{i\in\mathbb{N}}$ subconverges in pmG sense to 
the isomorphism class of a pmm space $\pmms$ satisfying the mixed curvature condition $RCD(K,N)+CAT(\bar K)$.
\end{corollary}

\subsection{Plan of the paper. }
In section 2 we recall preliminaries on convergence of metric measure spaces and various notions of convergence together with some general results. We  introduce the curvature-dimension condition $CD(\kappa, N)$ for general metric measure spaces and $\kappa\in C(X)$ and give a self-contained proof that this condition is equivalent to $\ric_M\geq \kappa$ for Riemannian manifolds. We also introduce the Riemannian curvature-dimension condition $RCD(K,N)$ for $K\in \R$. 

In section 3 we prove that uniform bounds on the integral curvature quanitity $k_{[X]}(\kappa,p,K)$ for $CD(\kappa,N)$ spaces with $p>\frac{N}{2}$ yields precompactness under measured Gromov-Hausdorff and measured Gromov convergence. A similar statement holds for the pointed case.

In section 4 we derive estimates for the case of $1$-dimensional metric measure spaces. 

In section 5 we present some technical obeservations that derive from the Area and the Co-Area formula. 

In section 6 we use the $1$-dimensional estimates and the technical lemma from the previous section to derive a displacement convexity inequality for smooth metric measure spaces with integral curvature bounds. 

In section 7 we prove the main theorem where we consider the cases $K\leq 0$ and $K>0$ separately. We finisch with a list of straightforward applications.

 \subsection{Acknowledgments} 
 I want to thank Robert Haslhofer for  useful comments and remarks on an earlier version of this paper. 
\section{Preliminaries}
\subsection{Metric measure spaces}  
\noindent
We follow \cite{gmsstability}.
Let $(X,\de_{\sX})$ be a complete and separable metric space. We denote by $\mathcal{M}_{loc}(X)$ the collection of Borel measures on $X$ which are finite on bounded sets,
by $\mathcal{M}(X)$ the subset of finite Borel measures, and by $\mathcal{P}(X)$ the collection of Borel probability measures.
We say a sequence $(\mu_i)_{i\in\mathbb{N}}$ of measures in $\mathcal{M}_{loc}(X)$ converges weakly to $\mu_{\infty}\in\mathcal{M}_{loc}(X)$ if 
\begin{align}\label{u01}
\lim_{i\rightarrow \infty}\int fd\mu_i=\int fd\mu_{\infty} \ \ \mbox{ for every }f\in C_{bs}(X)
\end{align}
where $C_{bs}(X)$ is the set of bounded continuous functions with bounded support.  If $(\mu_i)_{i\in \bar{\mathbb N}}\subset\mathcal{P}(X)$, then this is equivalent to require (\ref{u01}) with $f\in C_b(X)$, the set of bounded continuous functions.
\smallskip

Let $\m_{\sX}\in \mathcal{M}_{loc}(X)$. 
We call the triple $(X,\de_{\sX},\m_{\sX})$ a metric measure space (mm space).
The case $\m_{\sX}(X)=0$ is excluded. 
\smallskip

If $A\subset X$ is measurable with $\m_{\sX}(A)<\infty$, we set $\m_A:= \m_{\sX}|_A$ and $\bar{\m}_A={\m_{\sX}(A)}^{-1}\m_A$. 
If $\m_{\sX}(X)=1$, we say the mm space $X$ is normalized.
If we fix a point $o\in \supp\m_{\sX}$, we call $(X,o)$ a pointed metric measure space (pmm space).
\smallskip

Two mm spaces $X_i$, ${i=0,1}$, are called isomorphic if there exists an isometric embedding $\iota:\supp\m_{\sX_0}\rightarrow X_1$ such that
$\iota_{\star}\m_{\sX_0}=\m_{\sX_1}$. For pmm spaces $(X_i,o_i)$, ${i=0,1}$ we further require $\iota(o_0)=o_1$.
We shall denote by $\mmsc$ the corresponding isomorphism class of an mm space $X$, and with $\pmmsc$ the isomorphism class of a pmm space $(X,o)$.
\smallskip

The isomorphism class $[X]$  is invariant under $\m_{\sX}\mapsto r\cdot \m_{\sX}$, $r\in \mathbb R$. 
Hence, as  a represenative of $[X]$ we will usually pick one with $\m_X(X)=1$. In this case we say $X$ is normalized. Similar for the isomorphism class of $[X,o]$ we pick $(X,o)$ such that $\m_X(B_1(o))=1$. In this case we call $(X,o)$ normalized.
\subsection{Convergence of metric measure spaces}
Set $\bar{\mathbb{N}}=\mathbb{N}\cup\left\{\infty\right\}$.
We collect some  results on convergence of mm spaces that will be needed. 
\begin{definition}\label{mGHconvergence}
A sequence $(X_i,\de_{\sX_i})_{i\in\mathbb{N}}$ of compact metric spaces converges in Gromov-Hausdorff (GH) sense to a compact metric space $(X_{\sinfty},\de_{\sX_{\sinfty}})$ 
if 
there is a compact metric space $(Z,\de_{\sZ})$ and isometric embeddings $\iota_i:X_i\rightarrow Z$, $\iota:X\rightarrow Z$ such that $\iota_i(X_i)$ converges in Hausdorff sense to 
$\iota(X)$.
\smallskip

A sequence of compact mm spaces
$(X_i)_{i\in\mathbb{N}}$ with finite $\m_{\sX_i}$ converges in measured Gromov-Hausdorff (mGH) sense to a compact mm space \linebreak[4] $X_{\sinfty}$ 
if there exists a compact metric space $(Z,\de_{\sZ})$ and distance preserving embeddings $\iota_i,\iota: X_i, X\rightarrow Z$ as before such that the corresponding metric spaces converge in Gromov-Hausdorff sense and
$
(\iota_i)_{\star}\m_{\sX_i}\rightarrow (\iota)_{\star}\m_{\sX}
$
weakly in $\mathcal{M}(Z)$.
\end{definition}

\begin{definition}\label{pmGH}
We say pointed mm spaces $(X_i,o_i)$, $i\in \N$, converge in pointed measured Gromov-Hausdorff (pmGH) sense 
to a pointed mms space $(X_{\infty},o_{\infty})$ if for every $R>0$ and every $\epsilon>0$ there exists $i_{\sR,\epsilon}$ such that for $i\geq i_{\epsilon,\sR}$ there are measurable maps $f^{\sR,\epsilon}_i:X_i\rightarrow X_{\infty}$ such 
that
\begin{itemize}
\smallskip
 \item[(i)] $f^{\sR,\epsilon}_{i}(o_i)=o_{\infty}$,
\smallskip
 \item[(ii)] $\sup_{x,y\in B_{R}(o_i)}|\de_{\sX_i}(x,y)-\de_{\sX_{\infty}}(f^{\sR,\epsilon}_i(x),f^{\sR,\epsilon}_i(y))|<\epsilon$,
\smallskip
 \item[(iii)] $B_{\sR-\epsilon}(o_{\infty})\subset B_{\epsilon}(f^{\sR,\epsilon}_i(B_{R}(o_i)))$,
\smallskip
 \item[(iv)] $(f^{\sR,\epsilon}_i)_{\star}{\m}_{B_{R}(o_i)}$ converges weakly to ${\m}_{B_{R}(o_{\infty})}$ as $i\rightarrow \infty$.
\smallskip
\end{itemize}
\end{definition}
Let $(X, d_{\sX})$ be a metric space and $\epsilon>0$.  A subset $S\subset X$ is called  an $\epsilon$-net of $A\subset X$  if $A\subset \bigcup_{x\in S}B_{\epsilon}(x)$.

A family of metric spaces $\mathcal X$ is called uniformily totally bounded if  the following two statements hold. There exists $D$ such that for all $X\in \mathcal X$ $\diam_{X}\leq D$. For every $\epsilon>0$ there exists $N(\epsilon)\in \mathbb{N}$ such that every $X\in\mathcal X$ contains an $\epsilon$-net of not more than $N(\epsilon)$ points.

A family of pointed metric spaces $\mathcal X_0$ is called uniformily totally bounded if  for every $R>0$ and for every $\epsilon>0$ there exists $N(R,\epsilon)\in \mathbb N$ such that  the ball $B_R(o)$ admits an $\epsilon$-net of not more than $N(R,\epsilon)$ points for all $(X,d_{\sX}, o)\in \mathcal X_0$.
\begin{theorem}\label{th:compactness}
A sequence of  mm spaces $(X_i)_{i\in \mathbb N}$ such that the corresponding family of metric spaces is uniformily totally bounded and $\sup_{i\in \N}\m_{X_i}(X_i)\leq C<\infty$ admits a subsequence that converges in mGH sense to a  mm space $X_\sinfty$. 

A sequence of  pmm spaces $(X_i,o_i)_{i\in \mathbb N}$ such that the corresponding family of pointed metric spaces is uniformily totally bounded and $$\sup_{i\in \mathbb N}\m_{X_i}(B_R(o_i))\leq C(R)<\infty \ \ \forall R>1, $$  subconverges in  pmGH sense to  pmm space $(X_\sinfty,o_\sinfty)$.
\end{theorem}

\begin{definition}\label{mGconvergence}
A sequence of isomorphism classes
$\mmsic_{i\in\N}$ of mm spaces with finite $\m_{\sX_i}$ converges in measured Gromov (mG) sense to the isomorphism class $\mmsinfc$ of an mm space
if there exists a complete and separable metric space $(Z,\de_{\sZ})$ and isometric embeddings $\iota_i:X_i\rightarrow Z$, $\iota:X\rightarrow Z$ such that
$
(\iota_i)_{\star}\m_{\sX_i}\rightarrow (\iota)_{\star}\m_{\sX}$ weakly in $\mathcal{M}(Z)$.

In \cite{stugeo1} Sturm introduces a distance $\mathbb{D}$ on the space of isomorphism classes of normalized mm spaces that metrizes convergence in mG sense:  Consider $[X]$ and $[Y]$ for mm spaces $X$ and $Y$, normalize the measures $\m_{\sX}$ and $\m_{\sY}$ and define
\begin{align*}
\mathbb{D}([X],[Y])=\inf_{i_{\sX},i_{\sY}:X,Y\rightarrow Z}W_Z\left((\iota_{\sX})_{\#}\m_{\sX},(\iota_{\sY})_{\#}\m_{\sY}\right).
\end{align*}
The infimum is w.r.t. all distance preserving embeddings $\iota_{\sX}, \iota_{\sY}$ into a complete and separable metric space $(Z,d_{\sZ})$, and $W_Z$ is the Wasserstein distance in $(Z,d_{\sZ})$ that is introduced in the next subsection.
\end{definition}
\begin{definition}
Let $(X_{i},o_i)_{i\in\N}$ be a sequence of pmm spaces. We say that the corresponding sequences of isomorphism classes converges in pointed 
measured Gromov (pmG) sense to the
isomorphism class of a pmm space $(X_{\sinfty},o_{\sinfty})$ provided there exists a complete and separable metric space $(Z,\de_{\sZ})$ 
and isometric embeddings $\iota_{i}:{X_i}\rightarrow Z$ for $i\in\bar{\mathbb{N}}$ such that $(\iota_i)_{\star}\m_{\sX_i}\rightarrow (\iota_{\sinfty})_{\star}\m_{\sX_{\sinfty}}$ 
weakly in $\mathcal{M}_{loc}(X)$, 
and $\iota_i(o_i)\rightarrow \iota_{\sinfty}(o_{\sinfty})$ in $(Z,\de_{\sZ})$.
\end{definition}
\begin{theorem}[\cite{stugeo1}, \cite{gmsstability}]\label{th:gms}
If a sequence $\mmsi$ of mm spaces converges in mGH sense to a mm space $\mms$, then the corresponding equivalence classes converge in mG sense.

If pmm spaces $\pmmsi$ converge in pmGH sense 
to $\pmmsinf$ then the equivalence classes converge in pmG sense.
\end{theorem}
\begin{theorem}
If $\pmmsic_{i\in \N}$ converges in pmG sense to $\pmmscinf$
and $\pmmsi_{i\in \mathbb N}$ is  uniformly totally bounded, then $\pmmsi_{i\in\mathbb N}$ converges in pmGH
sense to a pmm space $(X,o)$ such that $\pmmsc=\pmmscinf$. 

In particular, if $\supp\m_{\sX_{\sinfty}}=X_{\sinfty}$, then $\pmmsi$ converges in pmGH to $\pmmsinf$
\end{theorem}
\subsection{Wasserstein space} Let $(X,\dex)$ be a complete and separable metric space.
The set of constant speed geodesics $\gamma:[0,1]\rightarrow X$ is denoted by $\mathcal{G}(X)$, 
and it is equipped with the topology of uniform convergence. $e_t:\gamma\mapsto \gamma(t)$ denotes the evaluation map at time $t$.
\smallskip

The $L^2$-Wasserstein space of Borel probability measures with finite second moment and the Wasserstein distance are denoted by 
$\mathcal{P}^2(X)$ and $W_{\sX}$, respectively. 
$\mathcal{P}^2_b(X)$ and $\mathcal{P}^2(\m_{\sX})$ denote the subset of probability measures with bounded support and
the family of $\m_{\sX}$-absolutely continuous probability measures, respectively.
\smallskip

We call a set $\Gamma\subset X^2$ $\frac{1}{2}d_{\sX}^2$-monotone or just monotone if for any finite collection $(x^1,y^1),\dots,(x^k,y^k)\in \Gamma$, $k\in \mathbb{N}$, we have 
\begin{align}\label{ineq:monotone}
\sum_{i=1}^k \frac{1}{2}d_{\sX}^2(x^i,y^i)\leq \sum_{i=1}^k\frac{1}{2}d_{\sX}^2(x^i,y^{\sigma(i)})
\end{align}
for any permutation $\sigma$ of $\left\{1,\dots,k\right\}$. 
The support $\supp\pi$ of an optimal coupling $\pi$ is a monotone set. If we replace $\frac{1}{2}d_{\sX}^2$ in \eqref{ineq:monotone} by a continuous function $c:X^2\rightarrow \mathbb{R}$, 
we call $\Gamma$ a $c$-monotone set.
\smallskip

A coupling or plan between probability measures $\mu_0$ 
and $\mu_1$ is a probability measure $\pi\in \mathcal{P}(X^2)$ such that $(p_i)_{\star}\pi=\mu_i$ 
where $(p_i)_{i=0,1}$ are the projection maps. We denote by $\Cpl(\mu_0,\m_1)$ the set of couplings between $\mu_0,\mu_1\in\mathcal{P}^2(X)$. A coupling $\pi\in \Cpl(\mu_0,\mu_1)$ is called optimal if 
$$
\int_{X^2}\dex(x,y)^2d\pi(x,y)=W_{\sX}(\mu_0,\mu_1)^2.
$$

A probability measure $\Pi\in \mathcal{P}(\mathcal{G}(X))$ is called an optimal dynamical coupling if $(e_0,e_1)_{\star}\Pi$ is an optimal coupling between its marginal distributions. 
Let $(\mu_t)_{t\in[0,1]}$ be a Wasserstein geodesic in $\mathcal{P}^2(X)$. We say an optimal dynamical coupling $\Pi$ is a lift of $\mu_t$ if $(e_t)_{\star}\Pi=\mu_t$ for every $t\in [0,1]$.
If $\Pi$ is the lift of a Wasserstein geodesic $\mu_t$, we call $\Pi$ itself a Wasserstein geodesic.
\medskip
\subsection{Curvature-dimension condition}
%\subsection{{Entropy functionals.}}
 Let $\mms$ be a mm space.
\smallskip
\noindent
Given $N\geq 1$ the $N$-R\'enyi entropy functional $S_{\sN}:\mathcal{P}^2(X)\rightarrow (-\infty,0)$ with respect to $\m_{\sX}$ is given by
\begin{align*}
\mu=\rho \m_{\sX}+\nu^*\mapsto S_{\sN}(\mu):=S_{\sN}(\mu|\m_{\sX})=-\int\rho^{1-\frac{1}{N}}(x)d\m_{\sX}.
\end{align*}
In the case $N=1$ the $1$-R\'eny entropy is $S_1(\mu)=-\m_{\sX}(\supp\rho)$. 
If $\m_{\sX}$ is finite, then 
$
-\m_{\sX}(X)^{\frac{1}{N}}\leq S_{\sN}(\cdot)\leq 0
$.
Moreover, if $N>1$
then $(\mu,\nu)\in \mathcal P(X)\times \mathcal P(X)\mapsto S_{\sN}(\mu|\nu)$ is lower semi-continuous w.r.t. weak convergence.
\medskip
\begin{definition}[generalized $\sin$-functions]\label{sin}
Let $\kappa:[0,L]\rightarrow \mathbb{R}$ be a continuous function.
The generalized $\sin$ function $\frs_{\kappa}:[0,L]\rightarrow \mathbb{R}$
is the solution of
\begin{align}\label{ode}
v''+\kappa v=0.
\end{align} 
such that $\frs_{\kappa}(0)=0$ and $\frs_{\kappa}'(0)=1$.
The generalized $\cos$-function is $\frc_{\kappa}=\frs_{\kappa}'$.
\end{definition}

\begin{definition}[Distortion coefficients]\label{generaldist}
Consider $\kappa:[0,{L}]\rightarrow \mathbb{R}$ that is continuous and $\theta \in [0,L]$. 
Then
\begin{align*}
\sigma_{\kappa}^{\sscr{(t)}}(\theta)=\begin{cases}
\frac{\frs_{{\kappa}}(t\theta)}{\frs_{{\kappa}}(\theta)}& \mbox{ if }\frs_{\kappa}(t)>0 \mbox{ for all }t\in (0,\theta]\\
\infty & \mbox{ otherwise }.
\end{cases}
\end{align*}
If $\sigma_{\kappa}^{\sscr{(t)}}(\theta)<\infty$, $t\mapsto\sigma_{\kappa}^{\sscr{(t)}}(\theta)$ is a solution of 
$
u''(t)+\kappa(t\theta)\theta^2u(t)=0
$
satisfying $u(0)=0$ and $u(1)=1$. We set $\sigma_{\kappa}^{\sscr{(t)}}(1)=\sigma_{\kappa}^{\sscr{(t)}}$ and
$\sigma_{\kappa}^{\sscr{(t)}}(\theta)=\sigma_{\kappa\theta^2}^{\sscr{(t)}}$.
We also define 
$
\pi_{\kappa}:= \sup\{ \theta\in (0,L]: \frs_\kappa(r)>0 \ \forall r\in (0,\theta]\}.
$\end{definition}

Consider a metric space $(X,\de_{\sX})$ and a continuous function $\kappa:X\rightarrow \mathbb{R}$.
We set $\kappa_{\gamma}=\kappa\circ\bar{\gamma}$ where $\gamma:[0,1]\rightarrow X$ is a constant speed geodesic and $\bar{\gamma}$ its unit speed reparametrization.
We denote by $\gamma^-(t)=\gamma(1-t)$ the reverse parametrization of $\gamma$, and we also write $\gamma=\gamma^{\sscr{+}}$, and $\kappa^{\sscr{-/+}}_{\gamma}:=\kappa_{\gamma^{-/+}}$.

\begin{proposition}\label{central}
Let $\VK:[a,b]\rightarrow \mathbb{R}$ be 
continuous and $u:[a,b]\rightarrow \mathbb{R}_{\geq 0}$ be an upper semi-continous. Then the following statements are equivalent:
\begin{itemize}
 \item[(i)]$u''+\VK u\leq 0$ in the distributional sense, that is 
\begin{align}\label{distributional}
\int_a^b\varphi''(t)u(t)dt\leq -\int_a^b\varphi(t)\VK(t)u(t)dt
\end{align}
for any $\varphi\in C_0^{\infty}\left((a,b)\right)$ with $\varphi\geq 0$.
\smallskip
 \item[(ii)] There is a constant $0<L\leq b-a$ such that
 \begin{align}\label{kuconcavity}
u(\gamma(t))\geq \sigma^{\sscr{(1-t)}}_{\VK^{\sscr{-}}_{\gamma}}(\theta)u(\gamma(0))+\sigma^{\sscr{(t)}}_{\VK^{\sscr{+}}_{\gamma}}(\theta)u(\gamma(1))
 \end{align}
for any constant speed geodesic $\gamma:[0,1]\rightarrow [a,b]$ with $\theta=|\dot{\gamma}|=\mbox{L}(\gamma)\leq L$. 
We set $\VK_{\gamma}=\VK\circ\bar{\gamma}:[0,\theta]\rightarrow \mathbb{R}$. $\bar{\gamma}:[0,\theta]\rightarrow [a,b]$ denotes the unit speed reparametrization of $\gamma$. We use the convention $\infty\cdot 0=0$.
\smallskip
\item[(iii)] The statement in (iii) holds for any geodesic $\gamma:[0,1]\rightarrow [a,b]$.
\end{itemize}
\end{proposition}

The modified distortion coefficient along $\gamma:[0,1]\rightarrow X$ w.r.t. $\kappa\in C(X)$ and $N\in [1,\infty)$ is given by
\begin{align*}
(t,\kappa,N)\mapsto\tau_{\kappa_{\gamma},N}^{\sscr{(t)}}(|\dot{\gamma}|)=
                                             t^{\frac{1}{N}}\left[\sigma_{\kappa_{\gamma}/(N-1)}(|\dot{\gamma}|)\right]^{1-\frac{1}{N}} & \mbox{ otherwise}
\end{align*}
where $r\cdot\infty=\infty$ for $r>0$, $0\cdot\infty=0$ and $\infty^{\alpha}=\infty$ for $\alpha>0$. By the Sturm-Picone comparison theorem
$\tau_{\kappa,N}^{\sscr{(t)}}(|\dot \gamma|)$ is non-decreasing in $\kappa$ and non-increasing in $N$.  

The following definition was introduced in \cite{ketterer5}.
\begin{definition}\label{bigg}
We say that an mm space $X\neq \{pt\}$ satisfies the \textit{curvature-dimension condition}
$CD(\VK,N)$ for $\kappa\in C_b(X)$ and $N\geq 1$ if for each pair $\nu_0,\nu_1\in \mathcal{P}_b^2(\m_{\sX})$ 
there exists an $L^2$-Wasserstein geodesic $(\nu_t)_{t\in[0,1]}\subset\mathcal{P}^2(\m_{\sX})$ and a dynamical optimal coupling $\Pi$ with $(e_t)_{\star}\Pi=\nu_t$ such that
\begin{align}\label{curvaturedimension}
S_{N'}(\nu_t)\leq{\textstyle -}\!\!\int\Big[\tau_{\VK^{\sscr{-}}_{\gamma},N'}^{\sscr{(1-t)}}(|\dot{\gamma}|)\varrho_0\left(e_0(\gamma)\right)^{-\frac{1}{N'}}+\tau_{\VK^{\sscr{+}}_{\gamma},N'}^{\sscr{(t)}}(|\dot{\gamma}|)\varrho_1\left(e_1(\gamma)\right)^{-\frac{1}{N'}}\Big]d\Pi(\gamma)
\end{align}
for all $t\in [0,1]$ and all $N'\geq N$ where $[\nu_i]_{ac}=\rho_i$, $i=0,1$.
\begin{comment}
\smallskip\\
We say a metric measure space $(X,\de_{\sX},\m_{\sX})$ satisfies the curvature-dimension condition $CD(\kappa,\infty)$ if the exponential growth condition (\ref{growthcondition}) holds and each pair $\nu_0,\nu_1\in \mathcal{P}^2(\m_{\sX})$
there exists an $L^2$-Wasserstein geodesic $\Pi$ of $\nu_0=\varrho_0\m_{\sX}$ and $\nu_1=\varrho_1\m_{\sX}$ such that
\begin{align}\label{curvaturedimension2}
\Ent(\mu_t)\leq (1-t)\Ent(\mu_0)+t\Ent(\mu_1)-\int_0^1 g(s,t)\kappa_{\Pi}(s\Theta)\Theta^2ds
\end{align}
for all $t\in [0,1]$ where $\mu_t=(e_t)_{\star}\Pi$. $g(s,t)=\min\left\{s(1-t),t(1-s)\right\}$ is the Green function of $(0,1)$ and $\Theta$ is the length of the geodesic $\mu_t$.
\smallskip\\
We say $\mms$ is a strong $CD(\kappa,N)$ space if (\ref{curvaturedimension}) for $N<\infty$ and (\ref{curvaturedimension2}) for $N=\infty$ hold for any $L^2$-Wasserstein geodesic.
\end{comment}
\end{definition}
For $x\rightarrow \kappa(x)=:K\in \mathbb{R}$ the definition is exactly the curvature-dimension condition as introduced by Lott-Sturm-Villani in \cite{stugeo1, stugeo2, lottvillani}.

As consequence of the monotinicity of the distortion coefficients we obtain the following property.
\begin{proposition}\label{htht}
Let $(X,\de_{\sX},\m_{\sX})$ be a metric measure space which satisfies the condition $CD(\VK,N)$ for a continuous function $\VK:X\rightarrow \mathbb{R}$ and $N\geq 1$. 

If $\VK':X\rightarrow \mathbb{R}$ is a continuous function such that $\VK'\leq \VK$, and if $N'\geq N$, then $(X,\de_{\sX},\m_{\sX})$ also satisfies the condition
$CD(\VK',N')$. 
\end{proposition}
Since we assume $\kappa\in C_b(X)$, the condition $CD(\kappa,N)$ implies the Lott-Sturm-Villani curvature-dimension condition $CD(K,N)$  for  $K\in \mathbb{R}$ with $\kappa\geq K$. In particular,  Bishop-Gromov volume growth estimate holds \cite{stugeo2,erbarkuwadasturm,ketterer5}, the space is locally compact and bounded sets have finite measure. 
Moreover, the $(\supp\m_{\sX},\de_{\supp\m_{\sX}})$ is a geodesic metric space.

\begin{proposition}\label{suddenlyimportant}
Let $X$ be a metric measure space which satisfies the condition $CD(\VK,N)$ for $\VK\in C_b(X)$ and $N\geq 1$. 
\begin{itemize}
\item[(i)]
If there is an isomorphism $\psi:(X,\de_{\sX},\m_{\sX})\rightarrow (X',\de_{\sX'},\m_{\sX'})$ onto a metric measure space $(X',\de_{\sX'},\m_{\sX'})$ then $(X',\de_{\sX'},\m_{\sX'})$
satisfies the condition $CD(\psi^{\star}\VK,N)$ with $\psi^{\star}\VK=\VK\circ\psi$.
\item[(ii)]
For $\alpha,\beta>0$ the rescaled metric measure space $(X',\alpha\de_{\sX'},\beta\m_{\sX'})$ satisfies $CD(\alpha^{-2}\VK,N)$.
\item[(iii)]
For each geodesically convex subset $U\subset X$ the metric measure space \linebreak[4] $(U,\de_{\sX}|_{U\times U},\m_{\sX}|_{U})$ satisfies $CD(\VK|_{U},N)$.
\end{itemize}
\end{proposition}

Let $(M,g_{M},e^{-V}\vol_M)$ be a weighted Riemannian manifold with $V\in C^{\infty}(M)$.
We recall that
for each real number $N>n$ the Bakry-Emery $N$-Ricci tensor is defined as
\begin{align*}
\ric^{N,V}(v)
&=\ric(v)-\nabla^2 V(v,v) - \frac{1}{N-n} \langle \nabla V, v\rangle^2
\end{align*}
where $v\in TM_p$. For $N=n$ we define
$$\ric^{N,V}(v):=
\begin{cases}
\ric(v)
+\nabla^2 V(v,v) &d V (v)=0\\
-\infty & \mbox{else}.
\end{cases}$$
For $1\leq N<n$ we define $\ric^{N,V}(v):=-\infty$ for all $v\neq 0$ and $0$ otherwise.

The following theorem appears in \cite{ketterer5}.
\begin{theorem}\label{smoothcase}
Let $(M,g_{\sM},e^{-V}d\vol_{\sM})$ be a weighted Riemannian manifold for $V\in C^{\infty}(V)$. Let $\VK:M\rightarrow \mathbb{R}$ be continuous and $N\geq 1$. 

Then, the mm space $(M,\de_{\sM},e^{-V}d\vol_{\sM})$ satisfies the condition $CD(\VK,N)$ if and only if it has $N$-Ricci curvature bounded from below by $\VK$.
\end{theorem}

\begin{proof} For completeness
we give a self-contained proof based on the Monge-Mather principle.
This allows us to introduce some concepts that will be used later again.
We assume familiarity with the concept of $\frac{1}{2}d^2$-concavity and Kantorovich duality.
\begin{theorem}[Monge-Mather principle, \cite{viltot} Corollary 8.2]
Let $(M,g_M)$ be a Riemannian manifold, and let $E\subset M$ be a compact subset.
Let $\gamma^1,\gamma^2:[0,1]\rightarrow E$ be two minimizing geodesics that satisfy 
\begin{align*}
d(\gamma^1(0),\gamma^1(1))^2+d(\gamma^2(0),\gamma^2(1))^2\leq d(\gamma^1(0),\gamma^2(1))^2+ d(\gamma^2(0),\gamma^1(1))^2.
\end{align*}
Then, there exists a constant $C_E>0$ such that for any $t_0\in (0,1)$ we have
\begin{align*}
d(\gamma^1(t),\gamma^2(t))\leq \frac{C_E}{\min(t_0,1-t_0)}d(\gamma^1(t_0),\gamma^2(t_0)), \ \forall t\in[0,1].
\end{align*}

\end{theorem} 

{\bf 1.}
We set $\m:= e^{-V}\vol_M$ and let $\mu_0,\mu_1\in \mathcal P_b^2(\m)$. We find $R>0$ such that $\mu_0$ and $\mu_1$ are supported in $B_{R/2}(o)$ for some $o\in M$. We set $B_R:= B_R(o)$ and $\bar B_R:= \bar B_R(o)$.

We first recall some general facts about $L^2$-optimal transport. 
 
There exists a dynamical optimal plan $\Pi\in \mathcal P(\mathcal G(X))$ such that $t\in [0,1]\mapsto (e_t)_{\#}\Pi=\mu_t$ is a $W_2$-geodesic in $\mathcal P^2(M)$.  Let $\supp\Pi =\Gamma$. The measures $\mu_t$ are supported in $(B_R)_t:=(e_t)(\Gamma)\subset B_R$.

{\bf 2.} {\it Claim.  $\mu_t\in \mathcal P^2(\m)$.}

We fix $t_0\in (0,1)$.
$(e_0,e_1)\circ e_{t_0}^{-1}((B_R)_{t_0})$ is a monotone subset of $B_{R/2}(o)\times B_{R/2}(o)$. 
Hence, for $t\in [0,1]$ the map $e_t\circ e_{t_0}^{-1}: (B_R)_{t_0}\rightarrow M$ is a Lipschitz continuous map by the Monge-Mather  principle. 
Let $t=0$. By the Kirzbraun theorem $e_{0}\circ e_{t_0}^{-1}$ admits a Lipschitz extension $T_{t_0}^0: \bar B_R(o)\rightarrow M$.

Then, the optimal plan $\pi_{t_0}^0=(e_{t_0},e_0)_{\#}\Pi$ between $\mu_{t_0}$ and $\mu_0$ is induced by the Lipschitz map $T_{t_0}^0$, that is $\pi^0_{t_0}=(\mbox{id}_{\bar B_R(o)}, T_{t_0}^0)_{\#} \mu_{t_0}$. Hence we get
\begin{align*}
\mu_{t_0}(A)= \pi_{t_0}^0(A,M)=\pi_{t_0}^0(A,T_{t_0}^0(A))\leq \pi_{t_0}^0(M,T_{t_0}^0(A))=\mu_0(T_{t_0}^0(A)).
\end{align*}
If $\mathcal R$ is a set with $\m(\mathcal R)=0$, then we get $\vol_M(\mathcal R)=0$. By Lipschitz continuity of $T_{t_0}^0$ it follows
 that $\vol_M(T_{t_0}^0(\mathcal R))=0$, hence $\m(T_{t_0}^0(\mathcal R))=\mu_0(T_{t_0}^0(\mathcal R))=0$ and therefore $\mu_{t_0}(\mathcal R)=0$. 
Since $t_{0}\in (0,1)$ was arbitrary so far, we  see that $\mu_{t_0}\in \mathcal{P}^2(\m)$ for all $t_0\in (0,1)$.

 {\bf 3.}
 There exists a $\frac{1}{2}d^2$-convex function $\phi:\bar B_R \rightarrow \R$ such that the following holds.
The Hamilton-Jacobi shift of $\phi$ is given by $$\phi_{t}(x)=\inf_{y\in M}\left\{\frac{1}{2t}d(x,y)^2-\phi(y)\right\},  \ \ \ x\in \bar B_R, \ t\in (0,1].$$
The function $\phi_{t}:\bar B_R\rightarrow \mathbb{R}$ is $\frac{1}{2t}d^2$-concave and hence Lipschitz continuous and semi-convex.
Let  $t_0\in (0,1)$ be as above. Then it is well-known \cite[Theorem 7.35]{viltot} that  the pair $(\phi_{t_0},\phi_s)$ is a solution for the dual Kantorovich problem w.r..t. $\mu_{t_0}$ and $\mu_s$ for $s\in [0,t_0)$, and similar $(-\phi_{t_0},\phi_s)$ is a solution of the dual Kantorovich problem w.r.t. $\mu_{t_0}$ and $\mu_s$ for $s\in (t_0,1]$. Let $M_{t_0}$ be the set of full $\m$-measure in $B_R$ where $\phi_{t_0}$ is differentiable. 
By a theorem of McCann  \cite{mccann}
$$\gamma_x(t)=\exp_x(-(t_0-t)\nabla\phi_{t_0}|_x)=:\tilde T_{t_0}^t(x), \ \  x\in  M_{t_0}$$
is the optimal map between $\mu_{t_0}$ and $\mu_s$ for all  $s\in [0,1]\backslash\{t_0\}$. 
By uniqueness of optimal maps it holds $T_{t_0}^t= \tilde T_{t_0}^t$ $\mu_{t_0}$-a.e. where $T_{t_0}^t=e_t\circ e_{t_0}^{-1}$ from before.  In particular, recall that  $T_{t_0}^t:\bar B_{R}(o) \rightarrow M$ is Lipschitz.

Since $\phi_t$ is semi-convex, there is a set $M_{t_0}'
\subset M_{t_0}$ of full $\m$-measure where $\phi_{t_0}$ is twice differentiable in Alexandrov sense, and $T_{t_0}^t$ admits a weak differential $DT_{t_0}^t$ on $M'_{t_0}$ for every $t\in [0,1]$.

{\bf 4.} {\it Monge-Amp\'ere inquality.}

By the area formula we obtain
\begin{align*}
\int_U\det DT_{t_0}^t(x)  \rho_t(T_{t_0}^t(x)) e^{-V(T_{t_0}^t(x))}d\vol_M(x) = \int_{T_{t_0}^t(U)} \mathcal{H}^0((T_{t_0}^t)^{-1}(x)) \rho_{t}(x) d\m(x)
\end{align*}
for every measurable subset $U$ and $t\in [0,1]$ and $t_0\in (0,1)$. On the other hand, by the measure theoretic transformation formula we deduce
\begin{align*}
 \int_{T_{t_0}^t(U)} \rho_{t}(x) d\m (x)= \int_{T_{t_0}^t(U)} d (T_{t_0}^t)_{\#}\mu_{t_0}(x)= \int_U d\mu_{t_0}(x)=\int_U \rho_{t_0}(x)e^{-V(x)}d\m(x).
 \end{align*}
Since $U$ was arbitrary, we obtain the Monge-Ampere inequality 
 \begin{align*}
 \rho_{t_0}(x)e^{-V(x)} \leq \det DT_{t_0}^t (x)  \rho_t(T_{t_0}^t(x))e^{-V(T_{t_0}^t(x))}\ \ \ \mbox{ for }x\in M_{t_0}'
 \end{align*}
 for $t_0\in (0,1)$ and $t\in [0,1]$ with equality if $t\in (0,1)$. 
 
 We set $DT_{t_0}^t(x)=:A_t(x)$.
  
{\bf 5.} Let us fix a transport geodesic $\gamma_x=T_{t_0}^t(x)$ for some $x\in M_{t_0}'$. The vector field $J_v: t\in [0,1]\rightarrow DT_{t_0}^t(x)v\in T_{\gamma_x(t)}M$ along $\gamma_x$ is smooth and satisfies the Jacobi equation for any vector $v\in T_{x}M$ with initial value $J_v(t_0)=v$. If $v\in (\gamma_x(t_0))^{\perp}$ then $J_v\in (\gamma_x(t))^{\perp}$ for every $t\in [0,1]$. 
Hence, one can define a symmetric linear map $B_x(t):= DT_{t_0}^t|_{\gamma_x(t_0)^{\perp}}: \gamma_x(t_0)^{\perp}\rightarrow \gamma_x(t)^{\perp}$. 
Then
Jacobi field calculus yields that $\log \det B_x(t)=:y_x(t)$ solves
\begin{align*}
y_x''(t)+ \frac{1}{n-1}(y_x'(t))^2 + \ric_{\gamma_x(t)}(\dot{\gamma}_x,\dot{\gamma}_x)\leq 0.
\end{align*}
Moreover, it is easy to check that $y_x(t)- v\circ \gamma_x(t)=:z_x(t)$  solves
\begin{align*}
z_x''(t) +\frac{1}{N-1}(z_x'(t))^2 + \ric^{N,V}_{\gamma_x(t)}(\dot{\gamma}_x,\dot{\gamma}_x)\leq 0.
\end{align*}
Hence, with $\ric^{N,V}\geq \kappa$ we deduce that $\det B_x(t)e^{-V\circ\gamma_x(t)}=:\mathcal I_x(t)$ solves
\begin{align*}
\frac{d^2}{dt^2} \mathcal I_x(t)^{\frac{1}{N-1}} + \frac{\VK}{N-1} |\dot{\gamma}_x|^2 \mathcal I_x(t)^{\frac{1}{N-1}}\leq 0
\end{align*}
and by Proposition \ref{central} it follows
\begin{align*}
\mathcal I_x(t)^{\frac{1}{N-1}}\geq \sigma_{\kappa_{\gamma_x}^{-}/(N-1)}^{\sscr{(1-t)}}(|\dot{\gamma}_x|) \mathcal I_x(0)^{\frac{1}{N-1}} + \sigma_{\kappa_{\gamma_x}^{+}/(N-1)}^{\sscr{(t)}}(|\dot{\gamma}_x|) \mathcal I_x(1)^{\frac{1}{N-1}}.
\end{align*}

{\bf 6.}
On the other hand, we can consider $\mathcal J_x(t)= \det A_t(x)e^{-V\circ\gamma_x(t)}$ and $L_x(t)= \mathcal J_x(t)/ \mathcal I_x(t)$. In \cite[Proof of Theorem 1.7 (c)]{stugeo2} it was shown that $L_x(t)$ is concave.
Following part (d) in the proof of Theorem 1.7 in \cite{stugeo2}  we obtain
\begin{align*}
\mathcal{J}_x(t)^{\frac{1}{N}}\geq \tau_{\kappa_{\gamma}^-,N}^{\sscr{(1-t)}}(|\dot{\gamma}_x|) \mathcal{J}_x(0)^{\frac{1}{N}} + \tau_{\kappa_{\gamma}^+,N}^{\sscr{(t)}}(|\dot{\gamma}_x|) \mathcal{J}_x(1)^{\frac{1}{N}}.
\end{align*}

{\bf 7.}
Thus together with the previous Monge-Ampere inequality and the measure theoretic change of variable formula we obtain
\begin{align*}
-S_N(\mu_t|\m)&= \int_M \left(\rho_t(x)\right)^{-\frac{1}{N}} d\mu_t(x) =\int_M \left(\rho_t(T_{t_0}^t(x))\right)^{-\frac{1}{N}} d\mu_{t_0}(x)\\
& \geq \int_M \det A_x(t)^{\frac{1}{N}} e^{-{\frac{1}{N}}V(T_{t_0}^t(x))}e^{\frac{1}{N}V(x)} \rho_{t_0}(x)^{1-\frac{1}{N}}d\m(x)\\
& = \int_M \mathcal{J}_x(t) e^{\frac{1}{N}V(x)} \rho_{t_0}(x)^{1-\frac{1}{N}}d\m(x)\\
& \geq \int_M \left[\tau_{\kappa_{\gamma}^-,N}^{\sscr{(1-t)}}(|\dot{\gamma}_x|) \rho_0(\gamma_x(0))^{-\frac{1}{N}} + \tau_{\kappa_{\gamma}^+,N}^{\sscr{(t)}}(|\dot{\gamma}_x|) \rho_1(\gamma_x(1))^{-\frac{1}{N}}\right] d\mu_{t_0}(x).
\end{align*}
Note that $\gamma_x(i)= e_i\circ e_{t_0}^{-1}(x)$, $i=0,1$. This proves one direction in the theorem.

{\bf 8.}
We skip  the proof of the  other direction since it will not play a role in the rest of the article.
\end{proof}

\begin{example}
Let $I\subset \mathbb{R}$ be an interval,
let $\VK:I\rightarrow \mathbb{R}$ be a lower semi-continuous function and let $u:I\rightarrow [0,\infty)$ be a non-negative solution
of
$
u''+ {\textstyle\frac{\VK}{N-1}} u=0
$ for $N>1$.
Then, the metric measure space
$
(I,|\cdot|_2,u^{N-1}d\mathcal{L}^1)
$
satisfies the curvature-dimension $CD(\VK,N)$.
\end{example}
\begin{remark}
The proof of the previous theorem actually shows more. If $(M,g_{M},\m)$ is a weighted Riemannian manifold, the condition $CD(\VK,N)$ holds iff for
each pair $\mu_0,\mu_1\in\mathcal{P}^2(\m_{\sX})$ with bounded support there exists a Wasserstein geodesic $\Pi$ with $(e_t)_{\star}\Pi=\mu_t$ such that
\begin{align}\label{something}
\varrho_t(\gamma_t)^{-\frac{1}{N}}\geq \tau_{\sk^-_{\gamma},\sN'}^{\sscr{(1-t)}}(|\dot{\gamma}|)\varrho_0(\gamma_0)^{-\frac{1}{N}}+\tau_{\sk^+_{\gamma},\sN}^{\sscr{(t)}}(|\dot{\gamma}|)\varrho_1(\gamma_1)^{-\frac{1}{N}}.
\end{align}
for all $t\in[0,1]$ and $\Pi$-a.e. $\gamma\in\mathcal{G}(X)$ where $\varrho_t\m_{\sX}=\mu_t$.
\end{remark}
\subsection{The measure contraction property}
\begin{definition}[\cite{ohtmea, stugeo2}]
We say a mm space $X$ satisfies the measure contraction property $MCP(K,N)$ for $K\in \R$ and $N\in (1,\infty)$ if for every $x_0\in X$ and
 $A\subset X$ ($A\subset B_{\pi_{K/(N-1)}}(x_0)$) with $\m(A)\in (0,\infty)$ there exists an optimal dynamical plan $\Pi$ such that $(e_0)_{\#}\Pi=\delta_{x_0}$, $(e_1)_{\#}\Pi=\m_A$ and 
\begin{align*}
\m \geq (e_t)_{\#}\left( \tau_{K,N}^{(t)}(|\dot{\gamma}|)^{N-1} \m(A) \Pi\right).
\end{align*}
\end{definition}
Ohta shows the following in \cite{ohtmea}.
\begin{theorem}
Let $X$ be a mm space that satisfies the $MCP(K,N)$ for $K>0$ and $N>1$. Then $\diam \supp\m_X\leq \pi_{K/(N-1)}$.
\end{theorem}
Let $X$ satisfy the $MCP(K,N)$ for $K>0$ and $N>1$. We say $y\in X$ is opposite to $x\in X$ if $d_X(x,y)=\pi_{\sK/(\sN-1)}$.
\begin{theorem}[Ohta, \cite{ohtmea}]
If $X$ satisfies the $MCP(K,N)$ for $K>0$ and $N>1$, then each point $x\in X$ has at most one opposite point.
\end{theorem}
\subsection{{Riemannian curvature-dimension condition}}
Let $(X,d_{\sX})$ be a metric space. For a function $u:X\rightarrow \mathbb{R}$ the local slope is 
\begin{align*}
\lip u(x)=\limsup_{y\rightarrow x}\frac{|u(x)-u(y)|}{d_{\sX}(x,y)}\  \ \mbox{if $x\in X$ is not isolated},
\end{align*}
and $
\infty$ otherwise. The Cheeger energy is defined in \cite{agslipschitz, agsheat} via
\begin{align*}
u\in L^2(\m_{\sX})\mapsto  \ChX(u)=\frac{1}{2}\inf\left\{\liminf_{h\rightarrow \infty}\int_{\sX}\left(\lip u_h\right)^2d\m_{\sX}: \left\|u_h-u\right\|_{L^2(\m_{\sX})}\rightarrow 0\right\}.
\end{align*}
Then the $L^2$-Sobolev space is $D(\ChX)=\left\{u\in L^2(\m_{\sX}):\ChX(u)<\infty\right\}$, and 
we say that a mm space $\mms$ is \textit{infinetimally Hilbertian} if the Cheeger enegery is quadratic. 
\begin{definition}[\cite{agsriemannian, giglistructure, erbarkuwadasturm, agmr, amsnonlinear, cavallettimilman}]
Let $K\in \mathbb{R}$ and $N\geq 1$. A metric measure space $\mms$ satisfies the \textit{
Riemannian curvature-dimension condition} $RCD(K,N)$ if $\mms$ is infinitesimally Hilbertian and satisfies the condition $CD(K,N)$.
\end{definition}
We say that an mm space $\mms$ is \textit{essentially nonbranching} if for any optimal dynamical coupling $\Pi$ there exists $A\subset \mathcal{G}(X)$ such that $\Pi(A)=1$ and for all $\gamma, \gamma'\in A$ we have that
$
\gamma(t)=\gamma'(t) \mbox{ for all }t\in [0,\epsilon] \mbox{ and for some }\epsilon>0\ \mbox{implies} \ \gamma=\gamma'.
$

In \cite{rajalasturm} Sturm and Rajala showed that $RCD(K,N)$-spaces are essentially non-branching.
\noindent
\begin{comment}
If $\mms$ is infinitesimally Hilbertian, the Cheeger enery $\ChX$ is a Dirichlet form on $L^2(\m_{\sX})$. Hence, there exists a semi-group $H_t:L^2(\m_{\sX})\rightarrow L^2(\m_{\sX})$ whose trajectories $t\mapsto f_t=H_tf$ 
are local Lipschitz curves into $L^2(\m_{\sX})$, and are the unique solution of 
\begin{align*}
Lf_t+\frac{d}{dt}f_t=0 \ \mbox{ a.e. in }(0,\infty), \mbox{ and }\lim_{t\downarrow 0}f_t=f \mbox{ in } L^2(\m_{\sX})
\end{align*}
where $L$ is the self-adjoint operator associated to $\ChX$.
\begin{theorem}
Let $\mms$ be a mm space satisfying $RCD(K,\infty)$. Then, if $\mu_0=f_0\m_{\sX}\in \mathcal{P}^2(X)$ with $f_0\in L^2$, then $\mathcal{H}_t\mu_0=(H_tf_0)\m_{\sX}$ for every $t\geq 0$.
\end{theorem}
\noindent
\end{comment}

\subsection{Integral Ricci curvature quantities}\label{subsec:ice}
Let $X$ be a mm space, let $\kappa:X\rightarrow \mathbb{R}$ be a continuous function such that $X$ satisfies $CD(\kappa,N)$. Consider
\begin{align*}
\left\|(\kappa-K)_-1_A\right\|_{L^p(\m_{\sX})}=\left(\int_A(\kappa-K)_{-}^pd{\m}_{\sX}\right)^{\frac{1}{p}}\in [0,\infty]. %
\end{align*}
where $(\kappa-K)_-=-\min(\kappa-K,0)$. If $\m_{\sX}(A)<\infty$, we set
\begin{align*}
\left\|(\kappa-K)_-\right\|_{L^p(\bar \m_A)}=\left(\int (\kappa-K)_{-}^pd\bar{\m}_{A}\right)^{\frac{1}{p}}\in [0,\infty].
\end{align*} 
In particular, if $\m_{\sX}(X)<\infty$,  for  the isomorphism class $[X]$ we define  $$k_{[\sX]}(\kappa,p,K)=(\diam X)^2\left\|(\kappa-K)_-\right\|_{L^p(\bar \m_{\sX})}.$$
We note that $k_{[X]}(\kappa,p,K)$ behaves naturally under scaling and  transformations by mm space isomorphisms. If $\psi: X\rightarrow X'$ is an mm space isomorphism, by Proposition \ref{suddenlyimportant} we have that $X'$ satisfies $CD(\kappa',N)$ for  $\kappa'=\kappa\circ \psi$ and we compute
\begin{align*}
\frac{\kappa_{[X]}(\kappa,p,K)^p}{(\diam X)^{2p}}&=\int (\kappa-K)_{-}^pd\bar{\m}_{X}=\int (\kappa\circ \psi -K)_{-}^p d\psi_{\#}\bar{\m}_X\\
&=\int (\kappa' -K)_{-}^pd\bar{\m}_{\sX'}=\frac{ \kappa_{[X']}(\kappa', p, K)^p}{(\diam X')^{2p}}.
\end{align*}
If $rX=(X, r d_{\sX}, \m_{\sX})$ is a rescaling of $X$ by $r\in \R$, one can check that $$k_{[X]}(\kappa,p,K)=k_{[rX]}({r}^{-1}\kappa, p, {r}^{-1}K).$$
Let $o\in X$.
Recall that for $[X,o]$  we pick the represenative such that $\m_X(B_1(o))=1$. We define
$${k}_{[\sX,o]}(\kappa,p,K,R)= R^2 \left\|(\kappa-K)_-1_{B_R(o)}\right\|_{L^p({\m_X})}.$$
We can check that $k_{[\sX,o]}(\kappa,p,K,R)$ again behaves naturally under   scaling and under pmm space isomorphisms.

\section{Precompactness}
\noindent
Let $\mms$ be a metric measure space that is essentially nonbranching and satisfies $CD(\kappa,N)$ for some admissible function $\kappa:X\rightarrow \mathbb{R}$. 
\medskip
\paragraph{\textbf{Spherical disintegration}} We 
fix a point $x_0\in X$, and 
consider the disjoint decomposition $X=\bigcup_{r>0}X_r$ where $X_r=\partial B_r(x_0)$.
According to this decomposition the measure $\m_{\sX}$ can be disintegrated as follows
\begin{align*}
\m_{\sX}=\int \bar{\m}_r\hspace{1pt}\sph(dr)
\end{align*}
where $\bar{\m}_r\in\mathcal{P}^2(X)$ and $\sph$ is a measure on $(0,\infty)$. The condition $CD(\kappa,N)$ implies that $r\mapsto\m_{\sX}(\bar{B}_r(x_0))$ is weakly differentiable. 
Hence, the measure $\sph$ is $\mathcal{L}^1$-absolutely continuous (for instance, see the 
proof of Theorem 5.3 in \cite{ketterer5}).
With slight abuse of notation we write $\sph(dr)=\sph(r)\mathcal{L}^1(dr)$ and set $\m_r=\sph(r)\bar{\m}_r$. Note that
\begin{align*}
 \m_{\sX}(\bar{B}_R(x_0))=\int_0^R\!\!\m_r(\bar{B}_R(x_0)) dr 
 =:v(R) \ \ \&\
 \ \ \frac{d}{dr}\m_{\sX}(\bar{B}_r(x_0))= \m_{r}(\bar{B}_r(x_0))=:s(r).
\end{align*}
\begin{example} Let $I_{\sK/(N-1)}$ be the model space $( [0,\pi_{\frac{K}{N-1}}], \sin_{\frac{K}{N-1}}r dr)$. In this situation we choose $x_0=0$, and we have
\begin{align*}
 \m_{\sK,\sN}(\bar{B}_R(x_0))=\int_0^R \sin_{\sK/(\sN-1)}^{\sN-1}r dr \ \ \ \&\ \ \ \frac{d}{dr}\m_{\sX}(\bar{B}_r(x_0))= \sin_{\sK/(\sN-1)}^{\sN-1}r.
\end{align*}
We set $v_{\sK,\sN}(r)=\m_{\sK,\sN}(\bar{B}_r(0))$ and $s_{\sK,\sN-1}(r)=\sin_{\sK/(\sN-1)}^{\sN-1}r$.
\end{example}
\begin{proposition}\label{volumegrowth}
Let $X$ be a mm space that is essentially nonbranching and satifies the condition $CD(\kappa,N)$. Let $x_0\in \supp\m_X$ and assume that $$\int_{B_D(x_0)} (\kappa -K)^p_{-} d\m<\infty$$ for some $p>\frac{N}{2}$, $K\leq 0$ and $D>0$.

Then, there exists a constant $C=C(K,N,p,D)>0$ such that 
\begin{align*}
\frac{d}{dR} \log \left(\frac{\m_{\sX}(B_{R}(x_0))}{\int_0^{R}w_{\sK,\sN-1}(r)dr}\right)\leq  {C}\left( \frac{R}{\m_X(B_R(x_0))}\int_{B_R(x_0)} (\kappa -K)^p_{-} d\m\right)^{\frac{1}{2p-1}}
\end{align*}
for every $R\in (0,D]$.
\end{proposition}
\begin{proof}\textbf{1.} Consider the sperical disintegration $\{\m_r\}_{r\in (0,\infty)}$ w.r.t. $x_0\in X$.
Let $R\in (0,\infty)$, and let $(\mu^{}_r)_{r\in[0,R]}$
be the constant speed $L^2$-Wasserstein geodesic that connects $\delta_{x_0}=\mu^{}_0$ and $\bar{\m}_R=\mu^{}_R$.
Let $\Pi^{}$ be the induced optimal dynamical plan. Then the $(e_t)_{\#}\Pi=:\mu_{tR}$ is supported on $\partial B_{tR}(x_0)$. 

Provided $X$ satisfies a curvature-dimension condition for $\kappa=const$ 
Cavalletti and Sturm prove that $\mu^{}_r$ is absolutely continuous w.r.t. $\m_{r}$ \cite[Lemma 3.2]{cavallettisturm}. It is easy to check that their argument also applies in our context.
We write $d\mu^{}_r=\hat{h}^{}_rd\m_{r}$.
We denote by $\Gamma^{}$ the support of $\Pi^{}$. $\Pi$-almost every $\gamma\in \Gamma$ has the length $R$ and we define $h_{r}:\Gamma\rightarrow \mathbb{R}$ with $h^{}_r(\gamma):=\hat{h}^{}_r(\gamma_r)$. 
Following again arguments of Cavalletti and Sturm \cite[Theorem 5.2]{cavallettisturm} one proves for $r_0,r_1\in (0,R)$
\begin{align}\label{a}
h^{}_{(1-t)r_0+tr_1}(\gamma)^{\frac{-1}{N-1}}\geq \sigma_{\kappa^-_{\sigma}/(N-1)}^{(1-t)}(|\dot{\varsigma}|)h^{}_{r_0}(\gamma)^{\frac{-1}{N-1}}+\sigma_{\kappa^+_{\sigma}/(N-1)}^{(t)}(|\dot{\varsigma}|)h^{}_{r_1}(\gamma)^{\frac{-1}{N-1}}
\end{align}
for $\Pi^{}$-a.e. $\gamma\in \mathcal{G}^{[0,R]}(X)$
where $\varsigma(t)=\gamma((1-t)r_0+tr_1)$.
\medskip

We set $r\in (0,R]\mapsto \omega(r):=h_{r}(\gamma)^{-1}$ and $g(r)=\log\omega(r)$. 
In particular, it follows $\omega(t)g'(t)=\omega'(t)$ and
(\ref{a}) is equivalent to 
\begin{align}\label{b}
g''+{\textstyle \frac{1}{N-1}}g^2+\kappa\circ \gamma\leq 0 \ \mbox{ on }\ (0,R]\ \mbox{ for }\ \Pi\mbox{-a.e. }\gamma.
\end{align}
In the following we omit the dependence on $\gamma\in \Gamma$.
We have equality in (\ref{a}) and (\ref{b}) if
$h_r(\gamma)^{-{1}}=:h_{r,\sK,\sN-1}(\gamma)^{-1}=s_{\sK,\sN-1}(r)$. In this case we write $g_{\sK,\sN-1}$ and $\omega_{\sK,\sN-1}$ for $g$ and $\omega$ respectively.  Then we compute
\begin{align}\label{ineq:important}
\frac{d}{dr}\frac{\omega(r)}{\omega_{\sK,\sN-1}(r)}&=\frac{\omega'(r)}{\omega_{\sK,\sN-1}(r)}-\frac{\omega_{\sK,\sN-1}'(r)\omega(r)}{\omega_{\sK,\sN-1}(r)^2}\nonumber\\
&= \Big[g'(r)-g'_{\sK,\sN-1}(r)\Big]\frac{\omega(r)}{\omega_{\sK,\sN-1}(r)}\leq \psi_{\sK,\sN-1}(r)\frac{\omega(r)}{\omega_{\sK,\sN-1}(r)}.
\end{align}
where $\psi_{\sK,\sN-1}(r):=\max\left\{0,g'(r)-g'_{\sK,\sN-1}(r)\right\}$. Note that $\psi$ depends also on $\kappa\circ\gamma$.
\smallskip\\
\textbf{2.}
Integration of (\ref{ineq:important}) w.r.t. $r$ from $r_0\in (0,R)$ to $R$ yields
\begin{align*}
 \frac{\omega(R)}{\omega_{\sK,\sN-1}(R)}- \frac{\omega(r_0)}{\omega_{\sK,\sN-1}(r_0)}\leq \int_{r_0}^{R} \psi_{\sK,\sN-1}(r)\frac{\omega(r)}{\omega_{\sK,\sN-1}(r)}dr.
\end{align*}
Since $K\leq 0$, we have that $r\in [0,\infty)\mapsto \omega_{\sK,\sN-1}(r)$ is monotone increasing (and positive). Hence, multiplication with $\omega_{\sK,\sN-1}(R)$ and $\omega_{\sK,\sN-1}(r_0)$ yields
\begin{align}\label{inequalitysome}
{\omega(R)}\omega_{\sK,\sN-1}(r_0)-{\omega(r_0)}\omega_{\sK,\sN-1}(R)
\leq  \omega_{\sK,\sN-1}(R)R\int_0^1\psi_{\sK,\sN-1}(\tau R){\omega(\tau R)} d\tau.
\end{align}
Recall that $$\int \omega(R)d\Pi=\m_R(X) \ \ \& \ \int\omega(\tau R)d\Pi=\int_{\left\{\hat h_{\tau R}>0\right\}}\hat h_{\tau R}(x)^{-1}d\mu_{\tau R}(x)\leq \m_{\tau R}(X).$$
Therefore, integration of \eqref{inequalitysome} w.r.t. $\Pi$ and application of Fubini's theorem yield
\begin{align*}
&\underbrace{\int {\omega(R)}d\Pi}_{=\m_R(X)}\omega_{\sK,\sN-1}(r_0)-\underbrace{\int {\omega(r_0)}d\Pi}_{\leq \m_{r_0}(X)}\omega_{\sK,\sN-1}(R)\\
&\hspace{2cm}\leq \omega_{\sK,\sN-1}(R) R\int\int_0^1\psi_{\sK,\sN-1}(\tau R){\omega(\tau R)} d\tau d\Pi.
\end{align*}
Oberserve
\begin{align*}
&R\int\int_0^1\psi_{\sK,\sN-1}(\tau R){\omega(\tau R)} d\tau d\Pi\\
&\leq \left(R \int\int_0^1\psi_{\sK,\sN-1}^{2p-1}(\tau R){\omega(\tau R)} d\tau d\Pi \right)^{\frac{1}{2p-1}} \left(R\int_0^1 \int \omega(\tau R)  d\Pi d\tau \right)^{1-\frac{1}{2p-1}}.
\end{align*}
where  $R\int_0^1 \int \omega(\tau R)  d\Pi d\tau \leq \int_0^R \m_r dr= \m(B_R(x_0))$. Hence
\begin{align*}
& \m_R(X)\omega_{\sK,\sN-1}(r_0)-\m_{r_0}(X)\omega_{\sK,\sN-1}(R)
\\
&
\hspace{1cm}\leq\omega_{\sK,\sN-1}(R)\m(B_{R}(x_0))^{1-\frac{1}{2p-1}}\left[\int_0^{R}\int \psi_{\sK,\sN-1}^{2p-1}(r)\omega(r) d\Pi dr\right]^{\frac{1}{2p-1}}.
\end{align*}
Another integration w.r.t. $r_0$ from $0$ to $R$ yields
\begin{align*}
&
\m_R(X)\int_0^R\omega_{\sK,\sN-1}(r_0)dr_0-\m_{\sX}(B_R(x_0)) \omega_{\sK,\sN-1}(R)\\
&\leq R\omega_{\sK,\sN-1}(R)\m(B_{R}(x_0))\left(\frac{1}{\m_{\sX}(B_R(x_0))}\int\int_0^{R}\psi_{\sK,\sN-1}^{2p-1}(r)\omega(r)drd\Pi\right)^{\frac{1}{2p-1}}. 
\end{align*}
The left hand side actually is $\left[\int_0^{R}w_{\sK,\sN-1}(r)dr\right]^2\frac{d}{dR}\frac{\m_{\sX}(B_{R}(x_0))}{\int_0^{R}w_{\sK,\sN-1}(r)dr}.$ Hence 
\begin{align*}
&\left(\frac{\m_{\sX}(B_{R}(x_0))}{{\scriptstyle\int_0^{R}w_{\sK,\sN-1}(r)dr}}\right)^{-1}\frac{d}{dR}\frac{\m_{\sX}(B_{R}(x_0))}{{\scriptstyle\int_0^{R}w_{\sK,\sN-1}(r)dr}}\\
&\  \leq \frac{R\omega_{\sK,\sN-1}(R)}{{\scriptstyle \int_0^{R}w_{\sK,\sN-1}(r)dr}}\left(\frac{1}{\m_{\sX}(B_R(x_0))}\int\int_0^{R}\psi_{\sK,\sN-1}^{2p-1}(r)\omega(r)drd\Pi\right)^{\frac{1}{2p-1}}
\end{align*} 
We estimate $\frac{R\omega_{\sK,\sN-1}(R)}{\int_0^R\omega_{\sK,\sN-1}(r)dr}$ by $\max_{R\in [0,D]}\frac{R\omega_{\sK,\sN-1}(R)}{\int_0^R\omega_{\sK,\sN-1}(r)dr}=:\Xi(K,N,D)$. 
\medskip\\
\textbf{3.} Recalling Proposition \ref{petersenweiaubry} below in the next section we obtain
\begin{align*}
&\frac{d}{dR}\log \left(\frac{\m_{\sX}(B_{R}(x_0))}{\int_0^{R}w_{\sK,\sN-1}(r)dr}\right)\\
&\leq C\left(\frac{1}{\m_X(B_R(x_0))}\int_0^R\int  \int_0^1R(\kappa(\gamma_{t R}-K)_{-}^p\omega(t R) dt d\Pi  dr\right)^{\frac{1}{2p-1}}
\end{align*}
where $C=C(K,N,p,D)=\Xi(K,N,D) C(p,N)^{\frac{1}{2p-1}}$.
Recall that $\omega(\tau)=h_{\tau}(\gamma)^{-1}$ and therefore
\begin{align*}
\int (\kappa(\gamma_{\tau R})-K)_{-}^p\omega(\tau R)d\Pi=
\int (\kappa-K)_{-}^p h_{\tau}^{-1} d(e_{\tau R})_{\star}\Pi(x)\leq
\int (\kappa-K)_{-}^p d\m_{\tau R}.
\end{align*}
Hence
\begin{align*}
\frac{d}{dR} \log \left(\frac{\m_{\sX}(B_{R}(x_0))}{\int_0^{R}w_{\sK,\sN-1}(r)dr}\right)\leq  {C}\left( \frac{R}{\m_X(B_R(x_0))}\int_{B_R(x_0)} (\kappa -K)^p_{-} d\m\right)^{\frac{1}{2p-1}}
\end{align*}
\end{proof}
\color{black}
\begin{remark}
In the case of $CD(K,N)$ where $K=const$ we get that $$\frac{d}{dr}\frac{\m_{\sX}(B_r(x_0))}{\int_0^r\sin_{\sK/(\sN-1)}^{\sN-1}\tau d\tau}\leq 0.$$ This is the classical Bishop-Gromov volume comparison. 
%\end{itemize}
\end{remark}
\begin{corollary}
Let $C$ be a mm space that is essentially nonbranching and satifies the condition $CD(\kappa,N)$.
Let $R>0$, $r\in (0,R)$ and $x_0\in \supp\m_{\sX}$ such that  $\left\| (\kappa- K)_{-}1_{B_R(x_0)}\right\|_{L^p(\m_X)}<\infty$ for some $p>\frac{N}{2}$, and $K\leq 0$. Then
\begin{align*}
\frac{\m_{\sX}(B_{R}(x_0))}{v_{K,N}(R)}\leq \frac{\m_{\sX}(B_{r}(x_0))}{v_{K,N}(r)} e^{ C(K,N,p,D) \left( \frac{R^{2p}}{\m_X(B_R(x_0))}\int_{B_R(x_0)} (\kappa -K)^p_{-} d\m\right)^{\frac{1}{2p-1}}}.
\end{align*}
\end{corollary}

\begin{corollary}\label{bishopgromov}
Let $(X,o)$ be a normalized pmm space that is essentially nonbranching and satifies the condition $CD(\kappa,N)$. Let $D\geq 1$ and $p>\frac{N}{2}$. Then
\begin{align*}
\frac{\m_{\sX}(B_{R}(x))}{v_{K,N}(R)}\leq \frac{\m_{\sX}(B_{r}(x))}{v_{K,N}(r)} e^{ C(K,N,p,D) \left( k_{[X,o]}(p,K,D)\right)^{\frac{p}{2p-1}}}.
\end{align*}
for each $x\in B_{\frac{R}{2}}(o)\cap \supp \m_X$ and $0<r<1\leq \frac{1}{2}R\leq 2 R\leq D$.
\end{corollary}
\begin{proof}
Let $x\in B_{R/2}(o)$. It holds
\begin{align*}
B_1(o)\subset B_{R/2}(o)\subset B_R(x) \subset B_{2R}(o) \subset B_D(o)
\end{align*}
Therefore 
\begin{align*}
 \frac{R^{2p}}{\m_X(B_R(x_0))}\int_{B_R(x_0)} (\kappa -K)^p_{-} d\m\leq  D^{2p}\int_{B_D(o)} (\kappa-K)^p_{-} d\m = k_{[X,o]}(p,K,D)^p.
\end{align*}
and together with the previous Corollary the claim follows.
\end{proof}

\begin{corollary}\label{pointedprecompactness}
Consider the family $\mathcal{X}(p,K,N)$ of isomorphism classes of essentially nonbranching pmm spaces $(X,o)$ satisfying a condition 
\linebreak[4] $CD(\kappa,N)$  for some $\kappa\in C(X)$
such that ${k}_{[X,o]}(p,K,D)\leq f(D)$ for all $[X,o]\in \mathcal X(p,K,N)$,  for all $D\geq 1$, $p>\frac{N}{2}$ and some function $f:[1,\infty)\rightarrow \mathbb R$.
Then $\mathcal{X}(p,K,N)$ is precompact in the sense of pmG convergence.
\end{corollary}
\begin{proof}
Consider $\pmmsc\in\mathcal{X}(p,K,N,k)$, and let $D>0$. Let $(X,o)$ be a normalized representative of $\pmmsc$. Let $0<r<1\leq \frac{1}{2}R\leq 2R \leq D$.
By the previous corollary
the number of $r$-balls with center in $B_{R/2}(o)$ is  bounded by constant that depends only on $K,N,p,r$ and $D\geq 2R$.
Hence the family of pmm spaces with isomorphism class contain in $\mathcal X(p,K,N)$ is uniformily totally bounded. Moreover 
\begin{align*}
\sup_{[X,o]\in \mathcal X(p,K,D)} {\m_{\sX}(B_{R}(x))}\leq \frac{v_{K,N}(R)}{v_{K,N}(r)} e^{ C(K,N,p,2R) \left( f(2R)\right)^{\frac{p}{2p-1}}}.
\end{align*}
Hence, the family of mm spaces such that the corresponding isomorphism class is contained in $\mathcal X(p,K,N)$ is 
precompact w.r.t. pmGH convergence according to Theorem \ref{th:compactness}.
Hence, the corresponding isomorphism classes are precompact w.r.t. pG convergence.
\end{proof}

\section{1-dimensional estimates}
The computations in the previous section motivate the following $1D$ estimates. 

Let $K\in \mathbb R$ and $N\in [2,\infty)$, and
let $\kappa: [0,L]\rightarrow \mathbb{R}$. 
We set $\omega(r)= \frs_{\kappa/(N-1)}^{N-1}(r)$.
The function $g(r)=\log \omega(r)$ solves
\begin{align}\label{ineq:ricatti}
g'' + \frac{1}{N-1} (g')^2 + \kappa= 0\ \mbox{ on } \ [0,\pi_{\kappa/(N-1)}].
\end{align}
Let $g_{\sK,\sN-1}: [0,L] \rightarrow \mathbb{R}$ be defined as $g_{\sK,N-1}=\log \frs_{K/(N-1)}^{N-1}$. The function $g_{\sK,\sN-1}$ is a solution of 
\begin{align*}
g_{\sK,\sN-1}'' + \frac{1}{N-1}(g_{\sK,\sN-1}')^2 + K=0 \ \mbox{ on }\ [0,\pi_{K/(N-1)}].
\end{align*}
Let $\psi_{\sK,\sN-1}(r)=\max \left\{0, g'(r)-g'_{\sK,\sN-1}(r)\right\}$ that solves 
\begin{align*}
\psi_{\sK,\sN-1}'+\frac{ \psi_{\sK,\sN-1}^2}{N-1}+ \frac{2\psi_{\sK,\sN-1} g_{\sK,\sN-1}}{N-1}\leq \kappa \ \ \& \ \ 
\lim_{r\rightarrow 0}\psi_{\sK,\sN-1}(r)=0.
\end{align*}

The following Theorem by Petersen/Wei/Aubry in \cite[Lemma 2.2]{petersenwei} and \cite[Lemma 3.1]{aubry} is fundamental in the theory of integral Ricci curvature bounds.
\noindent
\begin{proposition}\label{petersenweiaubry} Let $\omega$ and $\psi_{K,N-1}$ be as above.
Then
\begin{align*}
\psi_{\sK,\sN-1}(r_0)^{2p-1}\omega(r_0)\leq  C'(p,N)\int_0^{r_0} (\kappa(r)-K)_{-}^p\omega(r ) dr.
\end{align*}
 where $r_0\in [0,L\wedge {\pi_{\kappa/(N-1)}}]\cap [0,\frac{1}{2}\pi_{K/(N-1)}]$
and 
\begin{align*}
\frs_{\sK/(N-1)}(r_0)^{4p-n-1}
\psi_{\sK,\sN-1}(r_0)^{2p-1}\omega(r_0)\leq  C'(p,N)\int_0^{r_0} (\kappa(r)-K)_{-}^p\omega(r ) dr.
\end{align*}
where $r_0\in [0,L]\cap (\frac{1}{2}{\pi_{K/(N-1)}}, \pi_{K/(N-1)})$ and $C'(p,N)=(2p-1)^p\left(\frac{N-1}{2p-N}\right)^{p-1}$.
\end{proposition}
\begin{remark}
Note that 
$C'(p,2)=(2p-1)^p\left(\frac{1}{2p-2}\right)^{p-1}\rightarrow 1$ as $p\rightarrow 1$.
\end{remark}
\begin{remark}\label{rem:const}
For $K>0$, $\epsilon>0$, $L\in (0, \pi_{K/(N-1)}-\epsilon)$ and $r_0 \in (0,L]$ it holds that 
\begin{align*}
\psi_{\sK,\sN-1}(r_0)^{2p-1}\omega(r_0)\leq  C''(p,N,K,\epsilon)\int_0^{r_0} (\kappa(r)-K)_{-}^p\omega(r ) dr.
\end{align*}
where $C''(p,N,K,\epsilon)=\max\{ C'(p,N), {\frs_{\sK/(N-1)}(\epsilon)^{-4p+n+1}}\}$. 
We set 
\begin{align*}
C:= C(p,N,K,\epsilon):=\begin{cases} C'(p,N)\ &\ \mbox{ if }K\leq 0,\\
C''(p,N,K,\epsilon)\ &\ \mbox{ if }K>0.
\end{cases}
\end{align*}
\end{remark}
\begin{corollary}\label{thatisacor}For $K\in \mathbb R$, $N\in [2,\infty)$, $p>\frac{N}{2}$, $\epsilon>0$ and $\theta\in (0,L\wedge \pi_{\kappa/(N-1)})\cap (0,\pi_{K/(N-1)}-\epsilon)$ it holds
\begin{align*}
\int_0^\theta\psi_{\sK,\sN-1}(r)^{2p-1}\omega(r)dr\leq   C\theta\int_0^{\theta} (\kappa(r)-K)_{-}^p\omega(r ) dr.
\end{align*}
\end{corollary}
\begin{remark}\label{rem:A} Recall the distortion coefficient
$$r\in [0,\theta]\mapsto \sigma_{\kappa/(N-1)}^{(r/\theta)}(\theta)=\frac{\frs_{\kappa/(N-1)}(r)}{\frs_{\kappa/N-1)}(\theta)}$$
for $\theta\in [0,\pi_{\kappa/(N-1)})$ and $\infty$ otherwise.
For $p>\frac{N}{2}$ and $\theta\in (0,L\wedge \pi_{\kappa/(N-1)})\cap (0,\frac{1}{2}\pi_{K/(N-1)}-\epsilon)$  we have the
inequality
\begin{align}\label{ineq:some}
&\int_{t\theta}^{\theta} \psi_{\sK,\sN-1}(r)\sigma_{\kappa/(\sN-1)}^{(r/\theta)}(\theta)^{\sN-1}dr\leq \nonumber\\
&\  \left(C\theta^{2p-1}\int_0^\theta (\kappa(r)-K)_{-}^p\sigma_{\kappa/(\sN-1)}^{(r/\theta)}(\theta)^{N-1}dr \right)^{\frac{1}{2p-1}}\left(\int_t^{1}\sigma_{\kappa/(\sN-1)}^{(s)}(\theta)^{N-1}ds\right)^{\frac{2p-2}{2p-1}}.
\end{align}
Indeed, the transformation $t\mapsto t\theta$ and
H\"older's inequality yield
\begin{align*}
&\int_{t\theta}^{\theta}\psi_{\sK,\sN-1}(r)\frs_{\kappa/(\sN-1)}^{\sN-1}(r)dr\leq\int_{t}^1 \theta\cdot \psi_{\sK,\sN-1}(s\theta)\frs_{\kappa/(\sN-1)}^{\sN-1}(s\theta)ds\\
&\ \ \ \ \leq \left(\int_t^{1} \big(\theta\cdot \psi_{\sK,\sN-1}(s\theta)\big)^{2p-1}\frs_{\kappa/(\sN-1)}^{N-1}(s\theta)ds \right)^{\frac{1}{2p-1}}\left(\int_{t}^{1}\frs_{\kappa/(\sN-1)}^{N-1}(s\theta)ds\right)^{1-\frac{1}{2p-1}}\\
&\ \ \ \ \leq \left(\theta^{2p-2}\int_0^{\theta} \cdot \psi_{\sK,\sN-1}(r)^{2p-1}\frs_{\kappa/(\sN-1)}^{N-1}(r)dr \right)^{\frac{1}{2p-1}}\left(\int_{t}^{1}\frs_{\kappa/(\sN-1)}^{N-1}(s\theta)ds\right)^{1-\frac{1}{2p-1}}.
\end{align*}
Then, we can apply Corollary \ref{thatisacor} and devide by $\frs_{\sK/(\sN-1)}^{\sN-1}(\theta)$ to obtain \eqref{ineq:some}.
\end{remark}
\begin{lemma}\label{lem:B}
Let $K\in \mathbb{R}$, $N\geq 2$, $\kappa: [0,L]\rightarrow \mathbb{R}$, $N$ and $\epsilon>0$ be as before, and let $\theta\in (0,L\wedge {\pi_{K/(N-1)}}-\epsilon)$.
Then
\begin{align}\label{ineq:distcoef3}
\tau_{\sK,\sN}^{(t)}(\theta)\leq \tau_{\kappa,\sN}^{(t)}(\theta)+\Lambda^{\frac{1}{N}}\left[t\theta\int_t^1\psi_{\sK,\sN-1}(s\theta)\sigma_{\kappa/(\sN-1)}^{(s)}(\theta)^{\sN-1}ds\right]^{\frac{1}{N}}.
\end{align}
where the constant $\Lambda$ is given by
$$\Lambda:=\Lambda(K,N,\epsilon):=1+\max_{r\in [\frac{1}{2}\pi_{\sK/(\sN-1)},\pi_{\sK/(\sN-1)}-\epsilon]} \frac{1}{\frs_{\sK/(\sN-1)}(r)^{N-1}}.$$
Note that $\pi_{K/(N-1)}=\infty$ if $K\leq 0$. In this case we set $\Lambda=1$.
\end{lemma}
\begin{proof}
Recall that by definition $\tau_{\kappa,N}^{(t)}(\theta)=\infty$ if $\theta\geq \pi_{\kappa/(N-1)}$. Then the inequality holds. So we assume that $L<\pi_{\kappa/(N-1)}$.

We observe that $\log\frs_{\kappa/(\sN-1)}^{\sN-1}=g$
solves
\begin{align*}
\frac{d}{dr}\frs_{\kappa/(\sN-1)}^{\sN-1}=\frs_{\kappa/(\sN-1)}^{\sN-1}\frac{d}{dr}g.
\end{align*}
Then, we compute
\begin{align*}
\frac{d}{dr}\left[\frac{\frs_{\kappa/(\sN-1)}}{\frs_{\sK/(\sN-1)}}\right]^{\sN-1}
\leq \psi_{\sK,\sN-1}\left[\frac{\frs_{\kappa/(\sN-1)}}{\frs_{\sK/(\sN-1)}}\right]^{\sN-1}
\end{align*}
where $\psi_{\sK,\sN-1}=\max\left\{0,g'-g'_{\sK/(\sN-1)}\right\}$ as before.
Integration w.r.t. $r$ from $t\theta\in (0,\theta)$ to $\theta$ yields
\begin{align*}
\left[\frac{\frs_{\kappa/(\sN-1)}(\theta)}{\frs_{\sK/(\sN-1)}(\theta)}\right]^{\sN-1}- \left[\frac{\frs_{\kappa/(\sN-1)}(t\theta)}{\frs_{\sK/(\sN-1)}(t\theta)}\right]^{\sN-1}\leq \int_{t\theta}^{\theta} \psi_{\sK,\sN}(r)\left[\frac{\frs_{\kappa/(\sN-1)}(r)}{\frs_{\sK/(\sN-1)}(r)}\right]^{\sN-1}dr.
\end{align*}
Crossmultiplication gives
\begin{align*}
\sigma_{\sK/(\sN-1)}^{(t)}(\theta)^{\sN-1}-\sigma_{\kappa/(\sN-1)}^{(t)}(\theta)^{\sN-1}\leq \int_{t\theta}^{\theta}\psi_{\sK,\sN}(r)\left[\sigma_{\kappa/(\sN-1)}^{(r/\theta)}(\theta) \frac{\frs_{\sK/(\sN-1)}(t\theta)}{\frs_{\sK/(\sN-1)}(r)}\right]^{\sN-1} dr.
\end{align*}
Let us consider the left hand side of the previous inequality, and recall that $r\in [0,L\wedge \frac{\pi_{K/(N-1)}}{2}]\mapsto \frs_{\sK/\sN}(r)$ is monotone increasing. It follows that
\begin{align*}
 &\int_{t\theta}^{\theta}\psi_{\sK,\sN}(r)\left[\sigma_{\kappa/(\sN-1)}^{(r/\theta)}(\theta) \frac{\frs_{\sK/(\sN-1)}(t\theta)}{\frs_{\sK/(\sN-1)}(r)}\right]^{\sN-1} dr
 \leq \Lambda\int_{t\theta}^{\theta}\psi_{\sK,\sN}(r)\sigma_{\kappa/(\sN-1)}^{(r/\theta)}(\theta)^{\sN-1} dr.
\end{align*}
Therefore it follows
\begin{align*}
\sigma_{\sK/(\sN-1)}^{(t)}(\theta)^{\sN-1}\leq \sigma_{\kappa/(\sN-1)}^{(t)}(\theta)^{\sN-1}+\Lambda\theta\int_t^1\psi_{\sK,\sN-1}(s\theta)\sigma_{\kappa/(\sN-1)}^{(s)}(\theta)^{\sN-1}ds
\end{align*}
Multiplication with $t$ yields
%and therefore
\begin{align}\label{ineq:distcoef}
\tau_{\sK,\sN}^{(t)}(\theta)^{\sN}\leq \tau_{\kappa,\sN}^{(t)}(\theta)^{\sN}+t\theta \Lambda\int_t^1\psi_{\sK,\sN-1}(s\theta)\sigma_{\kappa/(\sN-1)}^{(s)}(\theta)^{\sN-1}ds.
\end{align}
Since $(\alpha + \beta)^{1/N}\leq \alpha^{1/N}+\beta^{1/N}$, $\alpha,\beta>0$, the claim follows.
\end{proof}
We note that 
\begin{align*}
t\int_t^1\sigma_{\kappa/(\sN-1)}^{(s)}(\theta)^{\sN-1}ds\leq \int_t^1\tau_{\kappa,\sN}^{(s)}(\theta)^{\sN}ds.
\end{align*}
Then, by Remark \ref{rem:A} and Lemma \ref{lem:B} we obtain the following.
\begin{corollary} \label{cor:disest} Let $K\in \mathbb{R}$, $N\geq 2$, $p>N/2$, $\kappa:[0,L]\rightarrow \mathbb R $, $\epsilon>0$ and $\theta\in (0,L\wedge \pi_{K/(N-1)}-\epsilon)$ as before.
Then
\begin{align*}
&\tau_{\sK,\sN}^{(t)}(\theta)- \tau_{\kappa,\sN}^{(t)}(\theta)\leq  \\
&\Lambda^{\frac{1}{N}}\left(Ct\theta^{2p}\int_0^1 (\kappa(s\theta)-K)_{-}^p\sigma_{\kappa/(\sN-1)}^{(s)}(\theta)^{N-1}ds \right)^{\frac{1}{N(2p-1)}}\left(\int_t^{1}{\tau_{\kappa,\sN}^{(s)}(\theta)^{N}ds}\right)^{\frac{2p-2}{N(2p-)}}.
\end{align*}

\end{corollary}

\section{An application of Area and Co-area formula}

For this section we consider the following setup.
Let $\Pi$ be a dynamical optimal plan with $(e_t)_{\#}\Pi = \mu_t$, $t\in[0,1]$ and $\mu_i\in \mathcal{P}^2(\m)$, $i=0,1$ such that $\mu_i$, $i=0,1$, are concentrated in $B_{R/2}(o)$ for some $o\in M$. Then $\mu_t$ is concentrated in $B_{R}(o)=:B_{R}$ for all $t\in [0,1]$. Let us fix $t_0\in (0,1)$. 

Recall from paragraph 3 in the proof of Theorem \ref{smoothcase} that there exists a Lipschitz map $e_{t_0}^{-1}: \bar B_R \rightarrow \mathcal G(M)$ such that $(e_0,e_1)\circ e_{t_0}^{-1}(B_R)\subset M^2$ is a monotone set,  $(e_{t_0}^{-1})_{\#}\mu_{t_0}=\Pi$ and  $e_{t_0}^{-1}(x)=\gamma_x(\cdot)=T_{t_0}^{(\cdot)}(x)$ is given by $\gamma_x(s)=\exp_x(-(t_0-s)\nabla \phi_{t_0}(x))$  for every $x\in M_{t_0}$
where $\phi_{t_0}:\bar B_R\rightarrow \mathbb \mathbb R$ is the Hamilton-Jacobi shift of a $\frac{1}{2}d^2$-convex function $\phi:\bar B_R\rightarrow \mathbb R$. The set $M_t$ is the set of differentiability of $\phi_t$.

Since the map $(e_0,e_1)\circ e_{t_0}^{-1}:M_t\rightarrow  M\times M$ is Lipschitz continuous, the function $$x\in \bar B_R\mapsto l_{t_0}(x):=L(e_{t_0}^{-1}(x))=L(\gamma_x)=d_M(\gamma_x(0),\gamma_x(1))=|\nabla\phi_{t_0}|(x)$$ is Lipschitz continuous as well.
Hence $U:= l_{t_0}^{-1}((0,\infty))\cap B_R$ is an open set and $\vol_M(U)>0$. 
We also recall from paragraph 3 in  the proof of Theorem \ref{smoothcase} that $M'_{t_0}\subset B_R$ where $\phi_{t_0}$ is twice differentiable, is a set of full $\m$-measure in $B_R$. In particular, for $x\in M'_{t_0}\cap U$ there exists a unique gradient $\nabla \phi_{t_0}(x)$ and $\nabla\phi_{t_0}(x)\neq 0$. $M'_{t_0}\cap U$ is a measurable set.

\begin{theorem} Let $M$ be an $n$-dimensional Riemannian manifold and let $\phi: M \rightarrow \mathbb R$ be Lipschitz continuous.
\begin{itemize}
\item[(i)] (Coarea formula): Let $A\subset M$ be $\mathcal H^n$-measurable. Then, 
$A\cap \phi^{-1}(y)$ is $\mathcal H^{n-1}$-measurable for $\mathcal L^1$-a.e. $y\in \mathbb{R}$, the map $y\mapsto \mathcal H^{n-1}(A\cap \phi^{-1}(y))$ is $\mathcal L^{1}$-measurable, and it holds
\begin{align*}
\int_A |\nabla \phi| (x) d\mathcal H^n(x) = \int_{\mathbb R} \mathcal H^{n-1}(A\cap f^{-1}(y)) d\mathcal L^1(y).
\end{align*}
\item[(ii)] (Level set integration): Let $g$ be $\mathcal H^n$-integrable. Assume $\essinf |\nabla \phi|>0$. Then,
$g|_{\phi^{-1}(y)}$ is $\mathcal H^{n-1}$-integrable for $\mathcal L^1$-a.e. $y\in \mathbb{R}$, and
\begin{align*}
\int_{\left\{ \phi\geq t\right\}} g(x) |\nabla \phi| (x)\mathcal H^n(x)= \int_{t}^\infty \left[\int_{\left\{\phi=s\right\}} g(x)d \mathcal H^{n-1}(x) \right] d s.
\end{align*}
\end{itemize}
\end{theorem}

We  apply the coarea formula to $\phi=\phi_{t_0}$ and 
$A=M'_{t_0}\cap U$. It follows that for $\mathcal L^1$-a.e. $a\in \mathbb{R}$ the set 
$\{\phi_{t_0}=a\}\cap M'_{t_0}\cap U$ is $\mathcal H^{n-1}$-measurable. In other words, there exists a set $Q\subset \mathbb R$ of full $\mathcal L^1$-measure such that $\{\phi_{t_0}=a\}\cap M'_{t_0}\cap U$ is $\mathcal H^{n-1}$-measurable for all $a\in Q$.

\begin{lemma}\label{lem:coarea} Let $g:M\rightarrow [0,\infty)$ be measurable. Then it holds
\begin{align*}
\int_{M_{t_0}'\cap U} g d\vol_M= \int_Q\int_{\{\phi_{t_0}=a\}} 1_{ M_{t_0}'\cap U} g \frac{1}{|\nabla \phi_t|} d\mathcal H^{n-1} da.
\end{align*}
\end{lemma}
\begin{proof}
We consider $U_{\eta}=\left\{|\nabla\phi_{t_0}|(x)> \eta\right\}$. Then by  level set integration we get
\begin{align*}
\int_{M_{t_0}'\cap U_\eta} g d\vol_M= \int_Q \int_{\{\phi_t=a\}}1_{ M'_{t}\cap U_\eta}g\frac{1}{|\nabla \phi_t|} d\mathcal H^{n-1} da.
\end{align*}
Since $U_\eta \uparrow U=\bigcup_{\lambda>0}U_\lambda$ for $\eta\downarrow 0$, by an application of the monotone convergence theorem we obtain the desired statement.
\end{proof}
\begin{lemma} 
Let $a\in Q$.
There exists a  countably $\mathcal H^{n-1}$-rectifiable set $\Sigma_a\subset \{\phi_{t_0}=a\}$ such that $\mathcal H^{n-1}((\left\{\phi_{t_0}=a\right\}\cap M'_{t_0}\cap U)\backslash \Sigma_a)=0$.
\end{lemma}
\begin{proof}
Recall the following theorem that appears in the appendix of \cite[Theorem 17]{mccannexistence} where the following theorem is stated for the case $M=\mathbb R^n$.

\begin{theorem}[Non-smooth implicite function theorem]\label{th:nsift}
Let $\phi:M \rightarrow \mathbb R$ be semi-convex. If $\nabla \phi(x_0)\neq 0$ for some $x_0\in M$ and $\phi(x_0)=a$, then there exists $\delta>0$ such that $\mathcal H^{n-1}(\left\{\phi=a\right\}\cap B_{\delta}(x_0))<\infty$ and $\left\{\phi=a\right\}\cap B_{\delta}(x_0)$ is countably $\mathcal H^{n-1}$-rectifiable.
\end{theorem}
\noindent
We pick $x_0\in \{\phi_{t_0}=a\}\cap M_{t_0}'\cap U$ for $a\in Q$. Since $x_0\in M'_{t_0}$, $\phi_{t_0}$ is twice differentiable in $x_0$. Therefore there exists
a unique, non-zero gradient $\nabla \phi_{t_0}(x_0)$. By Theorem \ref{th:nsift} for every $\delta>0$ that is sufficiently small $B_{\delta}(x_0)\cap \{\phi_{t_0}=a\}$ is $\mathcal H^{n-1}$-rectifiable. The collection of all such balls $B_\delta(x_0)$ with $x_0\in \{\phi_{t_0}=a\}\cap M_{t_0}'\cap U$ is a Vitali covering of $\{\phi_{t_0}=a\}\cap M_{t_0}'\cap U$.
Therefore, the Vitali covering theorem for $\mathcal H^{n-1}$ implies that we can choose a countable subfamily $\{B_{\delta_i}(x_i)\}_{i\in\mathbb{N}}$ that still covers $\{\phi_{t_0}=a\}\cap M_{t_0}'\cap U$ up to a set of $\mathcal{H}^{n-1}$-measure $0$. 
We define $\Sigma_a:=\bigcup_{i\in \mathbb{N}}B_{\delta_i}(x_i)$. By construction $\Sigma_a$ is countably $\mathcal H^{n-1}$-rectifiable. This yields the claim.
\end{proof}

\bigskip

We define the map $F:\Sigma_a\times [0,1]\rightarrow M$ via $(x,t)\mapsto \gamma_x(t)=T_{t_0}^t(x)$. Recall that $F(x,t_0)=x$ for $x\in \Sigma_a$.

\begin{lemma}
{$F$ is Lipschitz continuous. }
\end{lemma}
{\it Proof of the lemma.}
Not that $\{(\gamma_x(0),\gamma_x(1)): x\in \Sigma_a\}$ is a $\frac{1}{2}d^2$-monotone set. Hence,  we see by the Monge-Mather principle that 
\begin{align*}
d_{\sM}(F(x,t),F(y,s))
& \leq \frac{C_E}{\min(1-t_0,t_0)} d_{\sM}(x,y) + R |t-s|.
\end{align*}
This proves the claim.\qed 
\\

By Rademacher's theorem  there exists a measurable set $N\subset \Sigma_a\times [0,1]$ such that 
$\mathcal{H}^n_{\Sigma_a\times [0,1]}(N)=0$ and
the differential $DF(x,t)$ exists for $\forall (x,t)\in (\Sigma_a \times [0,1])\backslash N$. In particular,
$DF{(x,t)}v$ exists for every $(x,t)\in (\Sigma\times [0,1])\backslash N$ and $v\in T_x\Sigma_a$.

Recall that $M_{t_0}'$ is the set where $\phi_{t_0}$ is twice differentiable. If $x\in M_{t_0}'$ then
$DT_{t_0}^t(x)=A_x(t)$ exists for all $t\in [0,1]$. Hence
$$DF(x,t)v = DT_{t_0}^tv=A_x(t)v, \ \ v\in T_x\Sigma_a$$ $$\forall(x,t)\in
(\Sigma_a'\times [0,1])\backslash N=\mathcal S$$
where $\Sigma_a'=\Sigma_a\cap \{\phi_{t_0}=a\}\cap M_{t_0}'\cap U\subset \Sigma_a$. 

The vectorfield $t\in [0,1]\mapsto A_x(t)v$ is the Jacobi field $J$ with  initial conditions $J(0)=v$ and $J'(0)=\nabla^2\phi_{t_0} v$.
Moreover
$$DF(x,t)\partial_t= \dot{\gamma}_x(t)\ \forall (x,t)\in \mathcal S.$$

\begin{proposition}\label{prop:areaformula}
The following holds
\begin{align*}
\mathcal H^n_M \left( F(\mathcal S)\right)=
 \int_0^1 \int_{\Sigma_a'}|\det B_x(t)| |\dot \gamma_x(t)| d\mathcal H^{n-1}(x) dt 
\end{align*}
If $g$ is a $\vol_M$-integrable, non-negative function, it holds
\begin{align*} \int_{F(S)} g(p) d \vol_M(p)=
 \int_0^1 \int_{\Sigma_a'}g(\gamma_x(t))|\det B_x(t)| |\dot \gamma_x(t)| d\mathcal H^{n-1}(x) dt.
\end{align*}
\end{proposition}
\begin{proof}
Recall the area formula.
\begin{theorem}[Area formula] Let $\mathcal R$ be a $\mathcal H^n$-rectifiable set and let $F: \mathcal R \rightarrow  M$ be Lipschitz continuous. 
\begin{itemize}
\item[(i)]
Let $A\subset \mathcal R$ be $\mathcal H^n$-measurable. Then
\begin{align*}
\int_A JF(x) d\mathcal H^n(x) =  \int_{M} \mathcal H^0(A\cap F^{-1}(y)) d\mathcal H^n(y).
\end{align*}
\item[(ii)] (Change of variable formula): Let $g: \mathcal R\rightarrow \mathbb R$ be $\mathcal H^n$-integrable. Then 
\begin{align*}
\int g(x) JF(x) d\mathcal H^n(x) =  \int_{M} \left[ \sum_{x\in f^{-1}(y)} g(x)\right] d\mathcal H^n(y).
\end{align*}
\end{itemize}
$JF(x)= |\det DF(x)|$ denotes the Jacobian of the map $F$. 
\end{theorem}
In our case we have $\mathcal R= \Sigma_a\times [0,1]$, $F:=F(x,t)=\gamma_x(t)$ and $A=\mathcal S$. For $(x,t)\in \mathcal S$ we observe that 
$DF(x,t)v=DA_t(x)v$ and $DF(x,t)\frac{d}{dt}=\dot{\gamma}_x(t)$ are perpendicular since $DA_x(t)v$ is a Jacobi field along $\gamma_x$ with initial value $v\in T_x\Sigma_a\perp \gamma_x(t_0)$.
Hence 
$$JF(x,t)=|\det DF(x,t)|= |\det A_x(t)|=|\det B_x(t)||\dot{\gamma}_x(t)|$$
where $B_x(t)= A_x(t)|_{T_x\Sigma_a}: T_x\Sigma_a\rightarrow T_{\gamma_x(t)}\Sigma_a$.
By the area formula it follows
\begin{align*}
 \int_0^1 \int_{\Sigma_a'}|\det B_x(t)| |\dot \gamma_x(t)| d\mathcal H^{n-1}(x) dt&= \int_{\mathcal S} JF(x) d\mathcal H^{n}(x)\\
 &= \int_{M} \mathcal{H}^{0}(\mathcal{S}\cap F^{-1}(p)) d \vol_M(p).
\end{align*}
Of course the formula also holds when $\Sigma_a'=\emptyset$. In this case all integrals become $0$.
\\

Note that $e_{t_0}^{-1}(\Sigma_a)\subset e_{t_0}^{-1}(M)=\Gamma_a$ is $\frac{1}{2}d^2$--monotone and hence contained in the $\frac{1}{2}d^2$-differential of the $\frac{1}{2}d^2$-convex function $t\phi_t$. 
One of the key observations in \cite{cavdec} is that $\Gamma_a$ is a $d$-monotone set (\cite[Proposition 4.1]{cavdec}). And as corollary of this observation Cavalletti obtains that
the family of transport segments forms a partition of $e(\Gamma_a\times [0,1])\subset M$ up to a set $L$ of $\vol_M$-measure $0$.

Hence, we obtain
\begin{align*}
 \int_0^1 \int_{\Sigma_a'}|\det B_x(t)| |\dot \gamma_x(t)| d\mathcal H^{n-1}(x) dt= \int_{F(\mathcal S)\backslash L} d \vol_M.
\end{align*}
The second formula can be derived similar.
\end{proof}

\section{Displacement convexity}\label{section:displacement}
\begin{theorem}\label{resultA}
Let $(M,o)$ be a smooth, normalized, pointed metric measure space that satisfies $CD(\kappa,N)$ for $\kappa\in C(X)$ and $N\geq 2$. Set $B_R(o)=B_R$ for all $R>0$.
Let $K\in \R$, $p>{\textstyle \frac{N}{2}}$ and $R\geq 1$ such that $k_{[\sM,o]}(p,K,2R)<\infty.$
Let 
$\mu_0,\mu_1\in\mathcal{P}^2(\m_{\sM})$ with $\mu_i(B_{\sR})=1$ and $[\mu_i]_{ac}=\rho_i$, $i=0,1$.
Let $\Pi$ be the dynamical optimal plan between $\mu_0$ and $\mu_1$.  Let $\epsilon>0$ and assume $\Pi(\{\gamma\in \mathcal G(M): L(\gamma)\leq \pi_{K/(N-1)}-\epsilon\})=1$. Then
\begin{align}\label{inequalityA}
S_{\sN}((e_t)_{\#}\Pi)&\leq -\int\left[\tau_{\sK,\sN}^{(1-t)}(|\dot{\gamma}|)\rho_0(\gamma_0)^{-\frac{1}{N}}+\tau_{\sK,\sN}^{(t)}(|\dot{\gamma}|)\rho_1(\gamma_1)^{-\frac{1}{N}}\right]d\Pi(\gamma)\nonumber\\
&\hspace{0.7cm}+2\m_{\sM}(B_{2R}(o))^{\frac{1}{N}}\Lambda^{\frac{1}{N}}C^{\frac{1}{N(2p-1)}} k_{[M,o]}(p,K,2R)^{\frac{p}{N(2p-1)}} \ \forall t\in (0,1).
\end{align}
\end{theorem}
\begin{proof}
By the $CD(\kappa,N)$ condition there exists a Wasserstein geodesic $\Pi$ such that 
\begin{align}\label{ineq:curvcond}
\rho_t(\gamma_t)^{-\frac{1}{N}}\geq \tau_{\kappa_{\gamma}^-,\sN}^{\sscr{(1-t)}}(|\dot{\gamma}|)\rho_0(\gamma_0)^{-\frac{1}{N}}+\tau_{\kappa_{\gamma}^+,\sN}^{\sscr{(t)}}(|\dot{\gamma}|)\rho_1(\gamma_1)^{-\frac{1}{N}}
\ \mbox{ for $\Pi$-a.e. $\gamma\in \mathcal G(M)$}
\end{align} 
where $(e_t)_{\star}\Pi=\mu_t=\rho_t\m_{\sM}$ is concentrated in $B_{2R}$.
In the following we sometimes omit the dependence on $\gamma$ and  write $\kappa^{+/-}_{\gamma}=\kappa^{+/-}$. 
First we consider 
$
\tau_{\kappa_{\gamma}^+,\sN}^{\sscr{(t)}}(|\dot{\gamma}|)\rho_1(\gamma_1)^{-\frac{1}{N}}.$

 Recall that the unique Wasserstein geodesic $\mu_t$ is given by $\mu_t=(T_{t_0}^t)_{\#}\mu_{t_0}$ for $t_0\in (0,1)$ and $T_{t_0}^t(x)=\exp_x(-(t-t_0)\nabla\phi_{t_0}(x))$ for $x\in M_{t_0}$ where $\phi_{t_0}$ is the Hamilton-Jacobi shift of $\frac{1}{2}d^2$-convex function $\phi:M\rightarrow \R$ and $M_{t_0}$ is the set of points where $\phi_{t_0}$ is differentiable. The corresponding dynamical plan is given by $(e_{t_0}^{-1})_{\#}\mu_{t_0}=\Pi$. Since $\mu_i(B_R)=1$, $i=0,1$, it follows that $\mu_t(B_{2R})=1$ for all $t\in (0,1)$.

By Corollary \ref{cor:disest} we have
\begin{align}\label{ineq:xxx}
&\tau_{K,\sN}^{\sscr{(t)}}(|\dot{\gamma}|)\rho_1(\gamma_1)^{-\frac{1}{N}}-
\tau_{\kappa_{\gamma}^+,\sN}^{\sscr{(t)}}(|\dot{\gamma}|)\rho_1(\gamma_1)^{-\frac{1}{N}}\leq\nonumber\\
&\left(\frac{1}{\rho_1(\gamma_1)}\right)^{\frac{1}{N}}\left(\Lambda Ct|\dot\gamma|^{2p}\int_0^1 (\kappa(\gamma(s))-K)_{-}^p\sigma_{\kappa^+/(\sN-1)}^{(s)}(|\dot{\gamma}|)^{N-1}ds \right)^{\frac{1}{N(2p-1)}}\nonumber\\
&\ \ \ \ \ \ \ \ \ \ \ \ \ \ \ \ \times\left(\int_t^{1}\tau_{\kappa^+,\sN}^{(s)}(|\dot{\gamma}|)^{N}ds\right)^{\frac{2p-2}{N(2p-1)}}
\end{align}
 for $\Pi$-a.e. $\gamma\in \mathcal{G}(M)$ where $\Lambda$ and $C$ are the constants introduced in Section 3.

Integrating \eqref{ineq:xxx} w.r.t. $\Pi$ and using first  Jensen's and then H\"older's inequality yields
\begin{align*}
&\int \left[\tau_{K,\sN}^{\sscr{(t)}}(|\dot{\gamma}|)\rho_1(\gamma_1)^{-\frac{1}{N}}-
\tau_{\kappa_{\gamma}^+,\sN}^{\sscr{(t)}}(|\dot{\gamma}|)\rho_1(\gamma_1)^{-\frac{1}{N}}\right]d\Pi(\gamma)\\
& \leq \Lambda^{\frac{1}{N}} \bigg(C \underbrace{t\int |\dot \gamma|^{2p} \int_0^1 (\kappa(\gamma(s))-K)_{-}^p\sigma_{\kappa^+/(\sN-1)}^{(s)}(|\dot{\gamma}|)^{N-1}\rho_1(\gamma_1)^{-1}ds d\Pi(\gamma)}_{=:(I)}\bigg)^{\frac{1}{N(2p-1)}} \\
&\ \ \ \ \ \ \ \ \ \ \ \ \ \ \ \ \ \ \ \ \ \ \ \ \ \ \ \ \ \ \ \ \ \  \ \times \bigg(\underbrace{\int \int_t^{1}\tau_{\kappa^+,\sN}^{(s)}(|\dot \gamma |)^{N}\rho_1(\gamma_1)^{-1}ds d\Pi(\gamma)}_{=:(II)}\bigg)^{\frac{2p-2}{N(2p-1)}}
\end{align*}

In the following we will estimate $(I)$ and $(II)$ separately. First, we consider $(II)$. It follows by \eqref{ineq:curvcond} and Fubini's theorem that
\begin{align*}
(II)&=\int \int_t^{1}\tau_{\kappa^+,\sN}^{(s)}(|\dot \gamma|)^{N}\rho_1(\gamma_1)^{-1}ds d\Pi(\gamma)\leq \int_t^{1}\int \rho_s(e_s(\gamma))^{-1} d\Pi(\gamma)ds\\
& \overset{(e_s)_{\#}\Pi=\mu_s}{=} \int_0^{1}\int \rho_s(x)^{-1}d\mu_s(x) ds\  =  \int_t^1\m(\left\{\rho_s>0\right\})ds\leq  \m(B_{2R}(o)).
\end{align*}
Now, we consider $(I)$. We write $\Pi=(e_t^{-1})_{\#}\mu_t$ and $e_t^{-1}(x)=\gamma_x$. 
First, we decompose $\Pi$ into $\Pi^1+\Pi^2=\Pi$ where $\Pi^1= \Pi|_{\left\{\gamma\in \mathcal G(M): \dot{\gamma}\neq 0\right\}}$ and $\Pi^2=\Pi|_{\left\{\gamma\in \mathcal G(M):\dot \gamma=0\right\}}$.
$(I)$ is linear w.r.t. $\Pi$ and vanishes on $\Pi^2$. Therefore it is sufficient to assume $\Pi^1=\Pi$ and 
$\mu_t=(e_t)_{\#}\Pi$, $t\in (0,1)$,  is concentrated in the open set $U=U_t=l_{t}^{-1}\{(0,\infty)\}$ as was considered in the previous section.

From step 4 and 6 in  the proof of Theorem \ref{smoothcase} we have the Monge-Ampere equality 
\begin{align}\label{ineq:ma}
\rho_s(T_t^s(x))^{-1}&= \mathcal J_x(s)\rho_t(x)^{-1}\nonumber\\
&= \det A_x(s) e^{-V\circ T_t^s(x)} \rho_t(x)^{-1}\nonumber\\
&= \det B_x(s) e^{-V\circ T_t^s(x)} L_x(s) \rho_t(x)^{-1} \ \ \ \forall s,t\in (0,1), \ \forall x\in M_t'
\end{align}
and the inequality 
\begin{align}\label{ineq:ma2}
\rho_s(T_t^s(x))^{-1}\leq \det B_x(s) e^{-V\circ T_t^s(x)} L_x(s) \rho_t(x)^{-1} \ \ \ \forall t\in (0,1),\ s=0,1 \ \forall x\in M_t'.
\end{align}
By step 6 in the proof of Theorem \ref{smoothcase} it also holds
\begin{align}\label{ineq:yyy}
\det B_x(s) \geq \sigma_{\kappa^+/(\sN-1)}^{(\sscr{s})}(|\dot{\gamma}_x|)^{N-1} \det B_x(1)\ \ \ \& \ \ \ L_x(s)\geq s L_x(1) 
\end{align}$\forall s\in [0,1],\ \forall x\in M'_t$.

Therefore, after recalling that $\mu_t=(e_t)_\#\Pi$ and $e_t^{-1}(x)=\gamma_x(\cdot)$ with $\gamma_x(s)=\exp(-(s-t)\nabla \phi_t(x))$ for $x\in M_t$, 
we obtain the following key estimate
\begin{align*}
(I)=&t \int |\dot\gamma_x|^{2p} \int_0^1(\kappa(\gamma_x(s))-K)_{-}^p\sigma_{\kappa^+/(\sN-1)}^{(s)} (|\dot \gamma_{x}|)^{N-1} \rho_1(\gamma_x(1))^{-1}ds d\mu_t(x) \\
=&t \int |\dot\gamma_x|^{2p} \int_0^1(\kappa(\gamma_x(s))-K)_{-}^p\sigma_{\kappa^+/(\sN-1)}^{(s)} (|\dot \gamma_{x}|)^{N-1} \rho_1(T_t^1(x))^{-1}ds d\mu_t(x) \\
\leq & \int_{M_t'\cap U}\int_0^1(\kappa(\gamma_x(s))-K)_{-}^p  |\dot\gamma_x|^{2p}{\left\{\sigma_{\kappa^+/(\sN-1)}^{(s)}(|\dot \gamma_{x}|)^{N-1} \det B_x(1)e^{-V(T_{t}^1(x))}\right\}}
\\
&\ \ \ \ \ \ \ \ \ \ \ \ \ \ \ \ \ \ \ \ \ \ \  \ \ \ \ \ \ \ \ \ \ \ \  \ \ \ \  \ \ \ \ \ \ \ \ \ \ \times {tL_x(1)}  e^{V(x)}\rho_t(x)^{-1}ds d\mu_t(x) \\
\leq & \int_{M_t'\cap U}\left(\int_0^1(\kappa(\gamma_x(s))-K)_{-}^p\det B_x(s)e^{-V(\gamma_x(s))}ds\right)
\\
&\ \ \ \ \ \ \ \ \ \ \ \ \ \ \ \ \ \ \ \ \ \ \ \ \ \ \ \ \ \ \ \ \ \ \  \ \ \ \ \ \ \ \ \ \ \ \ \ \  \times 
|\dot\gamma_x|^{2p}L_x(t)e^{V(x)}\rho_t(x)^{-1} d\mu_t(x) \\
=& \int_{M_t'\cap U}1_{\{\rho_t>0\}}\left(\int_0^1(\kappa(\gamma_x(s))-K)_{-}^p \det B_x(s)e^{-V(\gamma_x(s))}ds\right)|\dot\gamma_x|^{2p} e^{V(x)} d\m(x)\\
=& \int_{M_t'\cap U}\left(\int_0^1(\kappa(\gamma_x(s))-K)_{-}^p  \det B_x(s)e^{-V(\gamma_x(s))}ds\right)
|\dot\gamma_x|^{2p} d\vol_M(x)
=(I)'
\end{align*}
where we used \eqref{ineq:ma2} for $s=1$ in the first inequality, \eqref{ineq:yyy} in the second inequality,  $\det  B_x(t)=L_x(t)=1$ in the third equality from below and \eqref{ineq:ma}  for the second equality from below.

We apply Lemma  \ref{lem:coarea}  to disintegrate $\vol_M$.  It holds
\begin{align*}
(I)'&= \int_Q \int_{\Sigma_a'}\left(\int_0^1(\kappa(\gamma_x(s))-K)_{-}^pe^{-V(\gamma_x(s))} \det B_x(s) ds\right)|\dot\gamma_x|^{2p}\\
&\ \ \ \ \ \ \ \ \ \ \ \ \ \ \ \ \ \ \ \ \ \ \ \ \ \ \ \ \ \ \times \frac{1}{|\nabla\phi_t|(x)} d\mathcal H^{n-1}(x) da.
\end{align*}
where $\Sigma_a':=M_t'\cap U\cap \{\phi_t=a\}\cap \Sigma_a$ as was introduced in the previous section.

Recall that $\dot{\gamma}_x(t)=\nabla \phi_t(\gamma_x(t))$ for $\Pi$-a.e. $\gamma\in \mathcal{G}(M)$. It follows that
\begin{align*}
(I)'&= \int_Q \int_{\Sigma_a'}\int_0^1(\kappa(\gamma_x(s))-K)_{-}^pe^{-V(\gamma_x(s))} \det B_x(s) |\dot{\gamma}_x|^{2p-1} d\mathcal H^{n-1}(x)ds da \\
&\leq  (2R)^{2p-2}\int_Q \int_0^1\int_{\Sigma_a'}(\kappa(\gamma_x(s))-K)_{-}^pe^{-V(\gamma_x(s))} \det B_x(s) |\dot{\gamma}_x|d\mathcal H^{n-1}(x)ds da \\
&=: (I)''.
\end{align*}
We can apply Proposition \ref{prop:areaformula}. It yields
\begin{align*}
&(I)''= (2R)^{2p-2}\int_Q \int_{F(\mathcal S)} (\kappa(x)-K)_{-}^pe^{-V(x)} d\vol_M(x) da\\
& \leq \frac{(2R)^{2p}}{R^2}\int_Q \int_{B_{2R}} (\kappa(x)-K)_{-}^pd\m(x) da = \frac{(2R)^{2p}}{R^2} \mathcal L^1(Q) \left\|(\kappa-K)_-\right\|_{L^p(\m|_{B_{2R}})}^p.
\end{align*}
It is not difficult to see that for $\phi_t:\bar B_{2R}\rightarrow \mathbb R$ $\frac{1}{2}d^2$-concave, it holds
 $\mathcal L^1(Q)\leq\mathcal L^1(\mbox{Im}\phi_t)\leq 4R^2$. It follows
\begin{align*}
(I)''\leq (2R)^{2p} \left\|(\kappa-K)_-\right\|_{L^p(\m|_{B_{2R}(o)})}^p
\end{align*}
We conclude that 
\begin{align*}
&\int \tau_{K,\sN}^{\sscr{(t)}}(|\dot{\gamma}|)\rho_1(\gamma_1)^{-\frac{1}{N}}d\Pi(\gamma)-
\int \tau_{\kappa_{\gamma}^+,\sN}^{\sscr{(t)}}(|\dot{\gamma}|)\rho_1(\gamma_1)^{-\frac{1}{N}}d\Pi(\gamma)\\
&\leq
 \left[\left(\Lambda{C}R^{2p}\int (\kappa-K)_{-}^p d\m\right)^{\frac{1}{2p-1}}\left(\m(B_{2R}(o)) \right)^{1-\frac{1}{2p-1}}\right]^{\frac{1}{N}}\\
 &\leq \left[\m(B_{2R}(o))\Lambda C^{\frac{1}{2p-1}} \left(R^2\left(\int (\kappa-K)_{-}^p d{\m}\right)^{\frac{1}{p}}\right)^{\frac{p}{2p-1}}\right]^{\frac{1}{N}}\\
 &=\m(B_{2R}(o))^{\frac{1}{N}}\Lambda^{\frac{1}{N}}C^{\frac{1}{N(2p-1)}} k_{[X,o]}(p,K,2R)^{\frac{p}{N(2p-1)}}.
\end{align*}
The same inequality holds in the case when we replace $\kappa^+_{\gamma}$ by $\kappa^-_{\gamma}$ and $t$ by $1-t$. Therefore
\begin{align*}
\int\rho_t^{-\frac{1}{N}}d\mu_t&\geq \int\left[\tau_{\sK,\sN}^{(1-t)}(|\dot{\gamma}|)\rho_0(\gamma_0)^{-\frac{1}{N}}+\tau_{\sK,\sN}^{(t)}(|\dot{\gamma}|)\rho_1(\gamma_1)^{-\frac{1}{N}}\right]d\Pi(\gamma)\\
&\hspace{2cm}-2\Lambda^{\frac{1}{N}}\m(B_{2R}(o))^{\frac{1}{N}}C^{\frac{1}{N(2p-1)}} k_{[X,o]}(p,K,2R)^{\frac{p}{N(2p-1)}}.
\end{align*}
That was to prove.
\end{proof}

Following the prove of Theorem \ref{resultA} we also obtain

\begin{theorem}
Let $M$ be a smooth, normalized mm space that satisfies $CD(\kappa,N)$ for $\kappa\in C(M)$ and $N\geq 2$. 
Let $K\in \R$ and $p>{\textstyle \frac{N}{2}}$ such that $k_{[\sM]}(p,K)<\infty$.
Let 
$\mu_0,\mu_1\in\mathcal{P}^2(\m_{\sM})$ with $[\mu_i]_{ac}=\rho_i$, $i=0,1$.
Let $\Pi$ be the dynamical optimal plan  between $\mu_0$ and $\mu_1$. Let $\epsilon>0$ and assume $\Pi(\{\gamma: L(\gamma)\leq \pi_{K/(N-1)}-\epsilon\})=1$. Then
\begin{align}
S_{\sN}((e_t)_{\#}\Pi)&\leq -\int\left[\tau_{\sK,\sN}^{(1-t)}(|\dot{\gamma}|)\rho_0(\gamma_0)^{-\frac{1}{N}}+\tau_{\sK,\sN}^{(t)}(|\dot{\gamma}|)\rho_1(\gamma_1)^{-\frac{1}{N}}\right]d\Pi(\gamma)\nonumber\\
&\hspace{1.5cm}+2C^{\frac{1}{N(2p-1)}} k_{[X,o]}(p,K)^{\frac{p}{N(2p-1)}} \ \forall t\in (0,1).
\end{align}
\end{theorem}
As an immediate corollary we also derive Corollary \ref{cor:bm}

\begin{remark}\label{rem:mcp}
The proof of Theorem \ref{resultA} also yields the following. Let $(M,o)$ be as in Theorem \ref{resultA}. Let $x_0\in B_{R}(o)$  and $\mu\in \mathcal P(\m_{\sM})$ be a measure supported in $B_{r}(x_0)$ for $r\in (0,\pi_{K/(N-1)}-\epsilon)\cap (0,R)$ and $\epsilon>0$. Let $\Pi$ be  the dynamical optimal plan such that $(e_1)_{\#}\Pi=\mu$ and $(e_0)_{\#}\Pi=\delta_{x_0}$. Then it holds
\begin{align}
S_{\sN}((e_t)_{\#}\Pi)&\leq -\int\tau_{\sK,\sN}^{(t)}(|\dot{\gamma}|)\rho_1(\gamma_1)^{-\frac{1}{N}}d\Pi(\gamma)\nonumber\\
&\hspace{0.7cm}+\m_{\sM}(B_{2R}(o))^{\frac{1}{N}}\Lambda^{\frac{1}{N}}C^{\frac{1}{N(2p-1)}} k_{[M,o]}(p,K,2R)^{\frac{p}{N(2p-1)}} \ \forall t\in (0,1).
\end{align}
\end{remark}
\section{Stability}
\begin{theorem}
Let $\{(M_i,o_i)\}_{i\in\mathbb{N}}$ be a sequence of smooth, normalized pmm spaces that satisfy the condition $CD(\kappa_i,N)$ for $\kappa_i\in C(X_i)$ 
such that
$$
k_{[M_i,o_i]}(\kappa_i, p,K,R)\rightarrow 0 \mbox{ when  } i\rightarrow \infty \ \ \forall R>0
$$
with $K\in \R$ and $p>\textstyle{\frac{N}{2}}$. 
Then $\{[M_i,o_i]\}_{i\in\mathbb{N}}$ subconverges in pmG sense to 
the isomorphism class of a pmm space $\pmms$ satisfying the condition $CD(K,N)$.
\end{theorem}
\begin{proof} 
We first consider the case $K\leq 0$. We set $(M_i.o_i)=(X_i,o_i)$, $i\in \mathbb N$.

\textbf{1.}
Our assumptions allow to extracting a subsequence $\pmmsi$ that converges in pmGH sense to a pmm space
$(X_\sinfty,o_{\sinfty})$. The corresponding isomorphism classes $[X_i,o_i]$ converge w.r.t. pmG convergence to $[X_\infty,o_\infty]$. Since $(X_i,o_i)$ are length spaces, $(X_{\sinfty},o_\sinfty)$ is a length space too.  In particular, $\bar B_R(o_i)$ converges in GH sense to $\bar B_R(o_\sinfty)$ for all $R>0$ (\cite{bbi}, 
Remark 3.2.9 in \cite{gmsstability}) and $[\bar B_R(o_i), \m_{B_R(o_i)}]$ converge to $[\bar B_R(o_\infty), \m_{B_R(o_\sinfty)}]$ in mG sense.

For $i\in \bar \N$ and $R>0$ we set $B_R^i=\bar B_R(o_i)$, $d^i_R= d^i_{\bar B_R(o_i)}$, $\m^i_R=  \m_{\bar B_R(o_i)}$, $\bar \m_R^i:= \m(\bar B_R(o_i))^{-1}\m_R^i$ and $\m_R^i(X_i)=\alpha_R^i$.
It holds $\sup_{i\in \mathbb{N}}\alpha_R^i<\infty$ and $\alpha_R^i\rightarrow \alpha_R^\sinfty$ as $i\rightarrow \infty$ for all $R>0$.

Let $(Z,d_{\sZ})$ be a compact metric spaces where mGH convergence of $(B^i_{2R})_{i\in \N}$ is realized. Then, also mGH convergence of $(B_R^i)_{i\in \N}$
is realized in $Z$.  
\smallskip\\
\textbf{2.} Consider  ${R=n\in\mathbb{N}}$.
Note that for $\mu\in\mathcal{P}^2(\m_{\sX_i})$ with $\mu(B_n^i)=1$ we have $\mu\in\mathcal{P}^2(\bar{\m}_{n}^i)$.
We denote by $\rho$ the density w.r.t. $\m_{\sX_i}$ and with $\bar{\rho}_n$ the density w.r.t. $\bar{\m}^i_n$ that is $\bar{\rho}_n=\m_{\sX^i}(B_n^i)\rho$.

Since $[B^i_n,\de^i_n,\bar{\m}_n^i]\rightarrow [B^{\infty}_n,\de_n^{\sinfty},\bar{\m}_n^{\sinfty}]$ in mG sense, we find optimal couplings $q^i_n$ between $\bar{\m}_n^i$ and $\bar{\m}^{\sinfty}_n$ such 
that $$W_{\sZ}(\bar{\m}_n^i,\bar{\m}_{n}^{\sinfty})^2=\int \de_{\sZ}^2(x,y)dq^i_n(x,y)=\de_{i,n}^2\rightarrow 0.$$
As in the proof of Theorem 4.20 in \cite{stugeo1} we define $Q^i_n:\mathcal{P}^2(\bar{\m}^{\sinfty}_n)\rightarrow\mathcal{P}^2(\bar{\m}^i_n)$ by
\begin{align*}
Q^i_n:\bar{\rho}_n\m_n^{\sinfty}\mapsto Q^i_n(\bar{\rho}_n)\m_n^i \ \mbox{ where }\ Q^i_n(\bar{\rho}_n)(y)=\int\bar{\rho}_n(x)Q^i_n(y,dx)
\end{align*}
where $Q^i_n(y,dx)$ is a disintegration of $q^i_n$ w.r.t. $\bar{\m}_{n}^i$. One can check that
\begin{align*}
S_{\sN}(Q_n^i(\mu)|\bar{\m}^i_n)\leq S_{\sN}(\mu|\bar{\m}^{\sinfty}_n)\ \ \& \ \ \de_{\sZ,\sW}(Q_n^i(\mu),\mu)\rightarrow 0\ \mbox{ as } \ i\rightarrow \infty.
\end{align*}
Similar, we define $P^i_n:\mathcal{P}^2(\bar{\m}^{i}_n)\rightarrow\mathcal{P}^2(\bar{\m}^{\sinfty}_n)$ by
\begin{align*}
P^i_n:\bar{\rho}_n\m_n^{i}\mapsto P^i_n({\bar{\rho}}_n)\m_n^{\sinfty} \ \ \mbox{ where } P^i_n({\bar{\rho}}_n)=\int\bar{\rho}_n(x)P^i_n(y,dx)
\end{align*}
where $P^i_n(y,dx)$ is a disintegration of $q^i_n$ w.r.t. $\m_{n}^{\infty}$. Again, one can check that
\begin{align*}
S_{\sN}(P_n^i(\mu)|\bar{\m}^{\sinfty}_n)\leq S_{\sN}(\mu|\bar{\m}^{i}_n)\ \ \& \ \ \de_{\sZ,\sW}(P_n^i(\mu),\mu)\rightarrow 0\ \mbox{ as } \ i\rightarrow \infty.
\end{align*}\smallskip

\noindent
\textbf{3.}
Let $\mu_j\in\mathcal{P}^2(\m_{\sX_{\sinfty}})$ with density $\rho_j$ and $\mu_j(B_n(o))=1$ for  $j=0,1$. 
Assume $0\leq \rho_j\leq r<\infty$. We will remove this assumption at the end of the proof.
In particular, $\mu_0,\mu_1\in \mathcal{P}^2(\bar{\m}_{n}^{\sinfty})$ and the  $\bar{\m}_n^{\sinfty}$-densities are $\bar{\rho}_{n,j}$, ${j=0,1}$. Clearly
$$\rho_j=\alpha^{\sinfty}_n\bar{\rho}_{n,j}\ \ \&\ \ \ \bar{\rho}_{n,j}=\frac{\alpha^{\sinfty}_{2n}}{\alpha^{\sinfty}_{n}}\bar{\rho}_{2n,j}.
$$
Consider $Q_n^i(\mu_j)=:\mu_{n,j}^i$, ${j=0,1}$, and its densities $Q_n^i(\bar{\rho}_{j,n})$ w.r.t $\bar{\m}_n^i$.
Let $\rho_{n,j}^i$ be the densitiy of $Q_n^i(\mu_j)$ w.r.t. $\m_{\sX_i}$. Then
\begin{align*}
\rho^i_{n,j}=\alpha_{n}^iQ_n^i(\rho_j)\ \mbox{ $\&$  }\ 
Q_n^i(\rho_j)=\frac{\alpha^i_n}{\alpha^i_{2n}}Q_{2n}^i(\rho_j).
\end{align*}
Let $\Pi^i$ be the unique optimal dynamical plan between $\mu_{n,j}^i$, $j=0,1$ with $(e_t)_{\#}\Pi^i=\mu_{n,t}^i\in \mathcal{P}(\bar{\m}_n^i)$. It holds that $\mu_{n,t}^i(B_n^i)=\mu_{n,t}^i(B_{2n}^i)=1$.
By Theorem \ref{resultA} for the $L^2$-Wasserstein geodesic $\Pi^i$ between 
$Q_n^i(\mu_j)_{j=0,1}$ with 
$(e_t)_{\star}\Pi^i=\rho^i_{t}\m_{\sM_i}=\mu^i_{t}$, $(e_t)_{\star}\Pi^i(B^i_{2n})=1$ and $(e_0,e_1)_{\star}\Pi^i=\pi^i$
such that
\begin{align}\label{ineq:ttt}
\int{\rho}^i_{t}(x)^{-\frac{1}{N}}d\mu^i_{t}(x)&\geq \int\left[\tau_{\sK, \sN}^{(1-t)}(|\dot{\gamma}|){\rho}^i_{n,0}(\gamma_0)^{-\frac{1}{N}}+\tau_{\sK,\sN}^{(t)}(|\dot{\gamma}|){\rho}^i_{n,1}(\gamma_1)^{-\frac{1}{N}}\right]d\Pi^i(\gamma)\nonumber\\
&\hspace{1.2cm}-2\Lambda^{\frac{1}{N}}(\alpha_{2n}^i)^{\frac{1}{N}}C^{\frac{1}{N(2p-1)}} k_{[M_i,o_i]}(p,K,2n)^{\frac{p}{N(2p-1)}}
\end{align}
\medskip\\
\textbf{4.}
The left hand side in the last inequality is $-S_{\sN}(\mu_{t}^i|\m_{\sM_i})$. Changing the reference measure yields 
$$-S_{\sN}(\mu_{t}^i|\m_{\sM_i})=-(\alpha^i_{2n})^{-\frac{1}{\sN}}S_{\sN}(\mu_{t}^i|\bar{\m}_{2n}^i).$$
Then, we map $\mu_t^i$ with $P^i_{2n}$ to $B_n^\sinfty$ and obtain 
\begin{align}\label{ineq:ii}
\left({\textstyle \frac{\alpha^i_{2n}}{\alpha^{\infty}_{2n}}}\right)^{-\frac{1}{\sN}}S_{\sN}(P^i_{2n}\mu_{t}^i|\m_{\sX^{\infty}})=(\alpha^i_{2n})^{-\frac{1}{\sN}}S_{\sN}(P^i_{2n}\mu_{t}^i|\bar{\m}_{2n}^{\sinfty})\leq (\alpha^i_{2n})^{-\frac{1}{\sN}}S_{\sN}(\mu_{t}^i|\bar{\m}_{2n}^i).
\end{align}
Similar, the first term in the right hand side of \eqref{ineq:ttt} equals 
\begin{align*}
(\alpha_{n}^i)^{-\frac{1}{N}}\int\left[\tau_{\sK,\sN}^{(1-t)}(d(x,y))Q_n^i(\rho_0)(x)^{-\frac{1}{N}}+\tau_{\sK,\sN}^{(t)}(d((x,y))Q_n^i(\rho_1)(x)^{-\frac{1}{N}}\right]d\pi^i(x,y)
\end{align*}
That is exactly $(\alpha_{n}^i)^{-\frac{1}{N}}$ times the expression $T^{(t)}_{\sK,\sN}(\pi^i|\bar{\m}_n^i)$ in step (v) of the proof of Theorem 3.1 in \cite{stugeo2}. In our context  the reference measure  is $\bar{\m}_{n}^i$. Therefore, following (v) in the proof of Theorem 3.1 in \cite{stugeo2} and using the maps $P^i_n$ and $Q^i_n$ we obtain 
\begin{align}\label{ineq:i}
(\alpha_{n}^i)^{-\frac{1}{N}}T^{(t)}_{\sK,\sN}(\pi^i|\bar{\m}_n^i)&\leq (\alpha_{n}^i)^{-\frac{1}{N}}T^{(t)}_{\sK,\sN}(\bar{\pi}^{i}|\bar{\m}_n^{\sinfty})+(\alpha_{n}^i)^{-\frac{1}{N}}\hat{C}W_{\sZ}(\bar{\m}_n^{\infty},\bar{\m}_n^i)\nonumber\\
&\leq \left({\textstyle \frac{\alpha_{n}^i}{\alpha_{n}^{\sinfty}}}\right)^{-\frac{1}{N}}T^{(t)}_{\sK,\sN}(\bar{\pi}^{i}|\m_{\sX_{\infty}})+(\alpha_{n}^i)^{-\frac{1}{N}}\hat{C}\de^2_{n,i}
\end{align}
if $i\geq i_{\epsilon}$ for $i_{\epsilon}$ sufficiently large. The constant $\hat{C}$ does not depend on $i$. The measure 
$\bar{\pi}^{i}\in \mathcal P(X_{\sinfty}^2)$ is a coupling between $\mu_0$ and $\mu_1$ (not necessarily optimal) such that $\bar{\pi}^{i}\rightarrow \pi$ 
weakly if $i\rightarrow \infty$ for an optimal coupling $\pi$ between $\mu_0$ and $\mu_1$.
Hence, \eqref{ineq:ttt}, \eqref{ineq:ii} and \eqref{ineq:i} together imply
\begin{align*}
\left({\textstyle \frac{\alpha^i_{2n}}{\alpha^{\infty}_{2n}}}\right)^{-\frac{1}{\sN}}S_{\sN}(P^i_{2n}\mu_{t}^i|\m_{\sX^{\infty}})&\leq \left({\textstyle \frac{\alpha_{n}^i}{\alpha_{n}^{\sinfty}}}\right)^{-\frac{1}{N}}T^{(t)}_{\sK,\sN}(\bar{\pi}^{i}|\m_{\sX_{\infty}})+(\alpha_{n}^i)^{-\frac{1}{N}}\hat{C}\de^2_{n,i}\\
&\hspace{0.5cm}+2\Lambda^{\frac{1}{N}}(\alpha_{2n}^i)^{\frac{1}{N}}C^{\frac{1}{N(2p-1)}} k_{[X_i,o_i]}(p,K,2n)^{\frac{p}{N(2p-1)}}.
\end{align*}
\textbf{5.} As in step (vi) in the proof of Theorem 3.1 in \cite{stugeo2} one checks that $P^i_{2n}\mu_{n,t}^i$ converges weakly to a probability $\mu_t$ for every $t\in [0,1]$, and $\mu_t$ is as
geodesic between $\mu_0$ and $\mu_1$ in $\mathcal{P}^2(X)$, 
and the limit is $\m_{\sX}$-absolutely continuous. Weak convergence of $\bar{\m}_n^i$ and $\bar{\m}_{2n}^i$ for $i\rightarrow \infty$ yields $\alpha_n^i\rightarrow \alpha_n^{\sinfty}$ for every $n\in\mathbb{N}$.
Moreover, by  Lemma 3.3 in \cite{stugeo2} we have
\begin{align*}
\limsup_{i\rightarrow\infty}T^{(t)}_{\sK,\sN}(\bar{\pi}^i|\m_n^i)\leq T^{(t)}_{\sK,\sN}(\pi|\m_n^i).
\end{align*}
and by lower semi-continuity of the $N$-Reny entropy
\begin{align*}
\liminf_{i\rightarrow \infty}S_{\sN}(P^i_{2n}\mu_{t}^i|\m_{\sX_{\infty}})\geq S_{\sN}(\mu_t|\m_{\sX_{\sinfty}}).
\end{align*}
Hence, letting $i\rightarrow \infty$
 and using again lower (upper) semi-continuity of 
$S_{\sN}(\mu_t|\m_{\sX_{\sinfty}})$ ($T_{\sK,\sN}^{\scriptscriptstyle{(t)}}(\pi|\m_{\sX_{\sinfty}})$, we obtain a geodesic 
$(\mu_t)_{t\in[0,1]}$ and a optimal coupling $\pi$ such that 
\begin{align*}
S_{\sN}(\mu_{t}|\m_{\sX_{\infty}})&\leq T^{(t)}_{\sK,\sN}({\pi}|\m_{\sX_{\infty}}).
\end{align*}
\noindent
\textbf{6.}
Finally, we want to remove the assumption that $\rho_j\in L^{\infty}(\m_{\sX_{\sinfty}})$ for $j=0,1$. Therefore, consider general probability measures $\mu_0,\mu_1\in\mathcal{P}^2(\m_{\sX_{\sinfty}})$ with 
densities $\rho_j$, and $\mu_j(B_n^{\sinfty}(o))=1$ for $j=0,1$.
Fix an arbitrary optimal coupling $\tilde{\pi}$ between $\mu_0$ and $\mu_1$, and set for $r\in (0,\infty)$
\begin{align*}
E_r:=\left\{(x_0,x_1))\in X_{\sinfty}^2: \rho_0(x_0)\leq r, \rho_1(x_1)\leq r\right\}, \ \alpha_r=\tilde{\pi}(E_r),\ \tilde{\pi}^r:=\frac{1}{\alpha_r}\tilde{\pi}|_{E_r}.
\end{align*}
The coupling $\tilde{\pi}^r$ is an optimal coupling between its marginals $\mu_0^r$ and $\mu^r_1$ such that for $j=0,1$
\begin{align}\label{Pos1}
W_{\sX_{\sinfty}}(\mu_j,\mu_j^r)\leq \epsilon \ \mbox{ if }r>0 \mbox{ sufficiently large}.
\end{align}
Depending on $r>0$ we can construct $\mu_t^r$ and $\pi^r$ as before such that 
\begin{align*}
S_{\sN}(\mu^r_{t}|\m_{\sX_{\infty}})&\leq T^{(t)}_{\sK,\sN}({\pi}^r|\m_{\sX_{\infty}}).
\end{align*}
From (\ref{Pos1}) we obtain - after  choosing subsequences - that $\mu_t^r$ converges weakly to a probabiltiy $\mu_t$ for $t\in [0,1]\cap\mathbb{Q}$. Then, again as in step (vi) of the proof of Theorem 3.1 in \cite{stugeo2}
$\mu_t$ extends to geodesic between $\mu_0$ and $\mu_1$ and $$\liminf_{i\rightarrow \infty}
S_{\sN}(\mu^r_{t}|\m_{\sX_{\infty}})\geq 
S_{\sN}(\mu_{t}|\m_{\sX_{\infty}})$$
for $t\in [0,1]$. 
\smallskip\\
Set 
$\tau_{\sK,\sN}^{{\scriptscriptstyle (1-t)}}(|\dot{\gamma}|)\rho_0(\gamma_0)^{-\frac{1}{N}}+
\tau_{\sK,\sN}^{{\scriptscriptstyle (t)}}(|\dot{\gamma}|)\rho_1(\gamma_1)^{-\frac{1}{N}}=\psi(\gamma)$. $\psi$ is integrable w.r.t. $\tilde{\pi}$, since the distortion coefficients are bounded
$\rho_0$ and $\rho_1$ are probability densities for $\mu_0$ and $\mu_1$ respectively, and $\tilde{\pi}$ is an coupling between $\mu_0$ and $\mu_1$.
Therefore, if we set $\pi^{\epsilon}=\alpha_r\pi^r+\tilde{\pi}|_{X^2\backslash E_r}$,
it follows that 
\begin{align*}
\lim_{\epsilon\rightarrow 0}\left|T_{\sK,\sN}^{(t)}(\pi^{\epsilon}|\m_{\sX_{\sinfty}})-T_{\sK,\sN}^{(t)}(\pi^{r}|\m_{\sX_{\sinfty}})\right|=0.
\end{align*}
Now, by compactness we can choose subsequence $\epsilon_i$ such that $\pi^{\epsilon_i}$ converges weakly to an optimal coupling $\pi$ between $\mu_0$ and $\mu_1$. Since $\pi^{\epsilon}$ is a coupling between $\mu_0$ and $\mu_1$ for every $\epsilon>0$, we can apply again Lemma 3.3 from \cite{stugeo2} for upper semi-continuity of $T^{(t)}_{\sK,\sN}(\pi|\m_{\sX_{\sinfty}})$ in $\pi$.
Hence
\begin{align*}
S_{\sN}(\mu_{t}|\m_{\sX_{\infty}})\leq \liminf_{i\rightarrow \infty}
S_{\sN}(\mu^{r(\epsilon_i}_{t}|\m_{\sX_{\infty}})\leq \limsup_{i\rightarrow\infty} T_{\sK,\sN}^{(t)}(\pi^{\epsilon_i}|\m_{\sX_{\sinfty}})\leq T_{\sK,\sN}^{(t)}(\pi|\m_{\sX_{\sinfty}}).
\end{align*}
This finishes the proof for $K\leq0$.
\\

For $K>0$ we first prove the following preliminary result.

\begin{proposition}
Let $\{(M_i,o_i)\}_{i\in\mathbb{N}}$ be a sequence of smooth, normalized pmm spaces that satisfy the condition $CD(\kappa_i,N)$ for $\kappa_i\in C(X_i)$ 
such that
$$
k_{[M_i,o_i]}(p,K,R)\rightarrow 0 \mbox{ as } i\rightarrow \infty \ \ \forall R>0
$$
with $K>0$ and $p>\textstyle{\frac{N}{2}}$. 
Then $\{[M_i,o_i]\}_{i\in\mathbb{N}}$ subconverges in pmG sense to 
the isomorphism class of a pmm space $\pmms$ satisfying the $MCP(K,N)$.
\end{proposition}
\noindent
{\it Proof of the Proposition.} 
We already know that a limit pmm space $(X_\sinfty, o_\sinfty)$ exists. 

We fix $R\geq \pi_{K/(N-1)}$ and $x_{\sinfty}\in B_R^{\sinfty}\subset X_{\sinfty}$, and let $\mu\in \mathcal P(\m_{\sX_\sinfty})$ such that for some $\epsilon>0$ it holds $\mu(B_{\pi_{K/(N-1)}-\epsilon}(x_\sinfty))=1$.

Let $(Z,d_Z)$ be a metric space  where the mGH convergence of $\bar B_{2R}^i$ to $\bar B_{2R}^{\infty}$ is realized. Then, we can find a sequence $x_i\in B_{2R}^i$ such that $x_i\rightarrow x_\sinfty$ in $Z$ and $\bar B_{\pi_{K/(N-1)}-\epsilon}(x_i)$ converges in Hausdorff
 sense to $\bar B_{\pi_{K/(N-1)}-\epsilon}(x_\sinfty)$ in $Z$.  In particular $[\bar B_{\pi_{K/(N-1)}-\epsilon}(x_i)]$ converges in mG sense. Moreover, for $i\in \N$
  sufficiently large it holds $\bar B_{\pi_{K/(N-1)}-\epsilon}(x_i)\subset B_{2R}^i$.
We set $\bar B_{\pi_{K/(N-1)}-\epsilon}(x_i)=B^i$, $\m^i=\m_{X_i}|_{B^i}$, $\bar \m^i= \m^i(B^i)^{-1}\m^i$ for $i\in \bar\N$.

As in step {\bf 2.} of the proof for $K\leq 0$ we construct $Q^i:\mathcal{P}^2(\bar \m^\sinfty)\rightarrow\mathcal{P}^2(\bar \m^i)$ by
\begin{align*}
Q^i:\bar{\rho}_n\m^{\sinfty}\mapsto Q^i(\bar{\rho})\m^i \ \mbox{ where }\ Q^i(\bar{\rho})(y)=\int\bar{\rho}(x)Q^i(y,dx)
\end{align*}
The measure $Q^i(\mu)=\mu^i$  is concentrated in $\bar B_{\pi_{K/(N-1)}-\epsilon}(x_i)$.  By Remark \ref{rem:mcp} it holds
\begin{align}\label{ineq:iii}
S_{\sN}((e_t)_{\#}\Pi^i)&\leq -\int\tau_{\sK,\sN}^{(t)}(|\dot{\gamma}|)\rho_1(\gamma_1)^{-\frac{1}{N}}d\Pi^i(\gamma)\nonumber\\
&\hspace{0.7cm}+\m_{\sM}(B_{2R}(o))^{\frac{1}{N}}\Lambda^{\frac{1}{N}}C^{\frac{1}{N(2p-1)}} k_{[M,o]}(p,K,2R)^{\frac{p}{N(2p-1)}} \ \forall t\in (0,1).
\end{align}
where $\Pi^i$ is the unique optimal dynamical plan between $\delta_{x_i}$ and $\mu^i$.

At this point we can essentially repeat the steps from before and it holds
\begin{align}\label{ineq:iv}
S_{\sN}((e_t)_{\#}\Pi^{\sinfty})&\leq -\int\tau_{\sK,\sN}^{(t)}(|\dot{\gamma}|)\rho_1(\gamma_1)^{-\frac{1}{N}}d\Pi^{\sinfty}(\gamma) \ \forall t\in (0,1).
\end{align}
for an optimal dynamical plan $\Pi^{\sinfty}$ between $\delta_{x_{\sinfty}}$ and $\mu$ that is the limit of $(\Pi^i)_{i\in\N}$.

By another limiting process for $\epsilon\downarrow 0$ we also obtain this inequality for measures $\mu$ with $\mu(B_{\pi_{K/(N-1)}}(x_\sinfty))=1$.

Inequality \eqref{ineq:iv} is a version of the measure contraction property that was used in \cite{cavmonoptmap}. It implies the measure contraction property in the sense of Ohta by work of Rajala \cite{rajala2}.
\end{proof}
\begin{corollary} Let $(M_i,o_i)$ be as in the previous proposition. 
For every $\epsilon>0$ there exists $i_\epsilon \in \N$ such that  $\diam{\supp\m_{\sM_i}}< \pi_{K/(N-1)}+\epsilon$ for all $i\geq i_\epsilon$.
\end{corollary}
\begin{proof}
We can argue by contradiction. Assume the statement is false. Then there exists a subsequence that is also denoted by $(M_i,o_i)$ and  points $x_i,y_i\in \supp\m_{M_i}$ such that $ d_{M_i}(x_i.y_i)>\pi_{K/(N-1)}+\epsilon$.
By the previous proposition we can extract a subsequence that converges in pmG convergence to 
some $MCP(K,N)$ space $X$. In particular, $d_{M_i}(x_i,y_i)\rightarrow d_X(x,y)$ for points $x,y\in X$. That is not possible because of the Bonnet-Myers theorem for $MCP(K,N)$ spaces \cite{ohtmea}.
\end{proof}
\smallskip
Now we can proceed with the proof of Theorem \ref{resultA}.
\smallskip

\noindent 
{\bf 7.} Let $\epsilon>0$.  By relabeling the sequence we can assume $\diam\m_{\sM_i}\leq \pi_{K/(N-1)}+\frac{1}{i}$. Let $i\geq i_{\epsilon}$  such that  $\frac{1}{i}\leq \epsilon<1$. We go back to step {\bf 2} from before and pick $\N\ni n=R\geq \pi_{K/(N-1)}+1$.   Then $ B^i_n=M_i$ and $\bar \m_n^i=\bar \m_{M_i}$. We replace $\m_{M_i}$ with its normalisation $\bar\m_{M_i}$ and we can assume that $[M_i]$ converges in Gromov sense to $[X_\sinfty]$ by means of the previous corollary.

We define the map $Q^i=Q^i_n:\mathcal P(\m_{X_\sinfty})\rightarrow \mathcal P(\m_{M_i})$ as in step {\bf 2} before.
We set $$\Xi(\epsilon):=\Xi^i(\epsilon):=\{\gamma\in \mathcal G(M_i): L(\gamma)\leq \pi_{K/(N-1)}-\epsilon\}.$$
Then we can follow the proof of the case $K\leq 0. $
As in {\bf 3} let $\mu_j\in\mathcal{P}^2(\m_{\sX_{\sinfty}})$ with density $\rho_j$ for  $j=0,1$. Assume $0\leq \rho_j\leq r<\infty$ on $\supp\mu_j$. We consider $Q^i(\mu_j)=\mu^i_j$ with density $\rho^i_j$ w.r.t. $\m_{M_i}$.
Then
\begin{align*}
\int(\rho^i_t)^{-\frac{1}{N}}d\mu_t
&\geq \int_{\Xi(\epsilon)}\left[\tau_{\sK,\sN}^{(1-t)}(|\dot{\gamma}|)\rho^i_0(\gamma_0)^{-\frac{1}{N}}+\tau_{\sK,\sN}^{(t)}(|\dot{\gamma}|)\rho^i_1(\gamma_1)^{-\frac{1}{N}}\right]d\Pi^i(\gamma)\\
&\hspace{2cm}-2\Lambda(\epsilon)^{\frac{1}{N}}C^{\frac{1}{N(2p-1)}} k_{[M_i]}(p,K)^{\frac{p}{N(2p-1)}}.
\end{align*}
where $\Pi^i$ is the dynamical optimal plan between $\mu_0$ and $\mu_1$, and $(e_t)_{\#}\Pi^i=\rho_t^i \m_{M_i}$. We could just drop the integral over $\mathcal G(M_i)\backslash \Xi(\epsilon)$ since it is non-negative.

We estimate the first term in the right hand side. Let $\epsilon_0\geq \epsilon$.  Since $K>0$, the map $\theta\mapsto \tau_{\sK,\sN}^{(t)}(\theta)$ is monoton increasing and it holds
\begin{align*}
&\int_{\Xi(\epsilon)}\tau_{\sK,\sN}^{(1-t)}(|\dot{\gamma}|)\rho^i_0(\gamma_0)^{-\frac{1}{N}}d\Pi^i(\gamma)\\
&\hspace{1cm}\geq  \int_{\Xi(\epsilon)}\tau_{\sK,\sN}^{(1-t)}(|\dot{\gamma}|\wedge(\pi_{K/(N-1)}-\epsilon_0))\rho^i_0(\gamma_0)^{-\frac{1}{N}}d\Pi^i(\gamma)\\
&\hspace{1cm}=\int\tau_{\sK,\sN}^{(1-t)}(|\dot{\gamma}|\wedge(\pi_{K/(N-1)}-\epsilon_0))\rho^i_0(\gamma_0)^{-\frac{1}{N}}d\Pi^i(\gamma)\\
&\hspace{2cm}-\tau_{\sK,\sN}^{(1-t)}(\pi_{K/(N-1)}-\epsilon_0)\int_{\mathcal G(M_i)\backslash \Xi(\epsilon)}(\rho_0^i)(\gamma_0)^{-\frac{1}{N}}d\Pi^i(\gamma)
\end{align*}
We define $\hat \Pi^i:=\hat \Pi^i_{\epsilon}= \Pi^i(\mathcal G(M_i)\backslash \Xi(\epsilon))^{-1}\Pi^i|_{\mathcal G(M_i)\backslash \Xi(\epsilon)}$. It holds 
$$(e_0)_{\#}\Pi^i= (e_0)_{\#}\Pi^i|_{\Xi(\epsilon)} + (e_0)_{\#}\Pi^i|_{\mathcal G(M_i)\backslash \Xi(\epsilon)}\geq (e_0)_{\#}\Pi^i|_{\mathcal G(M_i)\backslash \Xi(\epsilon)}=  \Pi^i(\mathcal G(M_i)\backslash \Xi(\epsilon))\hat \mu_0^i$$
where $\hat \mu_0^i=(e_0)_{\#}\hat \Pi^i$. $\mu_0^i$ is $\m$-absolutely continuous.
Hence $$\Pi^i(\mathcal G(M)\backslash \Xi(\epsilon))^{-1}\rho_0^i \geq \hat \rho_0^i \ \ \m_{M_i}\mbox{-a.e. }.$$
Therefore, it follows 
\begin{align*}
\int_{\mathcal G(M^i)\backslash \Xi(\epsilon)} (\rho_0^i)(\gamma_0)^{-\frac{1}{N}}d\Pi^i(\gamma)&=\Pi^i(\mathcal G(M_i)\backslash \Xi(\epsilon))\int (\rho_0^i)^{-\frac{1}{N}}
(\gamma_0)d\hat \Pi^i\\
&\leq \Pi^i(\mathcal G(M_i)\backslash \Xi(\epsilon))^{1-\frac{1}{N}}\int (\hat \rho_0^i)^{-\frac{1}{N}}
(\gamma_0)d\hat \Pi^i(\gamma)
\\
&
\leq - S_{\sN}((e_0)_{\#}\hat \Pi^i).
\end{align*}
We obtain

\begin{align*}
\int(\rho^i_t)^{-\frac{1}{N}}d\mu_t
&\geq \left[\tau_{\sK,\sN}^{(1-t)}(\hat \theta(\gamma))\rho^i_0(\gamma_0)^{-\frac{1}{N}}+\tau_{\sK,\sN}^{(t)}(\hat \theta(\gamma))\rho^i_1(\gamma_1)^{-\frac{1}{N}}\right]d\Pi^i(\gamma)\\
&\hspace{0.5cm}-2\Lambda(\epsilon)^{\frac{1}{N}}C^{\frac{1}{N(2p-1)}} k_{[M_i]}(p,K)^{\frac{p}{N(2p-1)}}+\hat C(\epsilon_0)
 S_{\sN}((e_0)_{\#}\hat \Pi^i)
\end{align*}
where $\hat \theta(\gamma) = |\dot{\gamma}|\wedge (\pi_{\sK/(\sN-1)}-\epsilon)$ and $\hat C(\epsilon_0)=\tau_{\sK,\sN}^{(t)}(\pi_{\sK/(N-1)}-\epsilon_0) $ for any $\epsilon_0\geq \epsilon$.
\smallskip\\
{\bf 8.}
Let $\epsilon_i = \frac{1}{i}$.
We study the sequence of dynamical optimal  plans $\hat \Pi^i=:\hat \Pi^i_{\epsilon_i}$. After going to a subsequence $\hat \Pi^i_{\epsilon_i}$ converges weakly in $\mathcal P(\mathcal G(Z))$ to a dynamical optimal plan $\hat \Pi\in \mathcal P(X_\sinfty)$.  By weak convergence $\hat \Pi$ is supported on $\Xi_0=\{\gamma\in \mathcal G(X_\sinfty): |\dot\gamma|\geq \pi_{\sK/(\sN-1)}\}$. Now $\frac{1}{2}d^2_{X_\sinfty}$-monotonicity yields
\begin{align*}
2\pi_{\sK/(\sN-1)}^2=d(\gamma_0,\gamma_1)^2+d(\hat\gamma_0,\hat \gamma_1)^2\leq d(\gamma_0,\hat\gamma_1)^2+d(\hat\gamma_0,\gamma_1)^2\leq 2\pi_{\sK/(\sN-1)}^2
\end{align*}
for any pair of geodesics $\gamma, \hat\gamma\in \Xi_0$.
Therefore, any coupling between $\hat \mu_0=(e_0)_{\#}\hat \Pi$ and $\hat \mu_1=(e_1)_{\#}\hat \Pi$ is optimal and supported on $\Xi_0$. By uniqueness of opposite points we conclude that $\hat\mu_j=\delta_{x_j}$ for opposite points $x_0$ and $x_1$. 
Hence $S_{\sN}(\hat \mu_j|\m_{X_\sinfty})=0$, $j=0,1$,  and by lower semi-continuity of $S_{\sN}$ it follows
\begin{align*}
\liminf_{i\rightarrow \infty} S_{\sN}((e_0)_{\#}\hat \Pi^i_{\epsilon_i}|\m_{M_i})\geq 0
\end{align*}
\smallskip
\noindent 
{\bf 9.} At this point we can follow verbatime the proof for $K\leq 0$ to  obtain the following estimate for $\mu_0$, $\mu_1$ and dynamical optimal plan $\Pi$ in $X_\sinfty$: 
\begin{align*}
S_{\sN}((e_t)_{\#}\Pi|\m_{\sX_\sinfty})\leq - \int \left[ \tau_{\sK,\sN}^{(1-t)}(\hat \theta(\gamma))\rho_0(\gamma_0)^{-\frac{1}{N}} + \tau_{\sK,\sN}^{(t)}(\hat \theta(\gamma))\rho_1(\gamma_1)^{-\frac{1}{N}}\right]d\Pi(\gamma)
\end{align*}
where $\hat \theta(\gamma)=|\dot\gamma|\wedge (\pi_{\sK/(\sN-1)}-\epsilon_0)$. To finish the proof we use monotonicity of $\theta\mapsto \tau_{\sK,\sN}^{(t)}(\theta)$ for $K>0$ and the monotone convergence theorem to let $\epsilon_0\rightarrow 0$.

\begin{theorem}\label{th:noncol}
Let $\{(M_i,o_i)\}_{i\in\mathbb{N}}$ be a sequence of $n$-dimensional, pointed Riemannian manifolds that satisfy the condition $CD(\kappa_i,N)$ for $\kappa_i\in C(X_i)$ 
such that 
$$
k_{[M_i,o_i]}(p,K,R)\rightarrow 0 \mbox{ when  } i\rightarrow \infty \ \ \forall R>0
$$
with $K\in \R$,  $p>\textstyle{\frac{N}{2}}$, $p\geq 1$ if $N=2$ and \eqref{noncollapsing} holds.
Then $\{[M_i,o_i]\}_{i\in\mathbb{N}}$ subconverges in pmG sense to 
the isomorphism class of a pmm space $\pmms$ satisfying the Riemannian curvature-dimension condition $RCD(K,N)$.
\end{theorem}
\begin{proof}
We need to check that $X$ is infinitesimal Hilbertian. By \cite[Proposition 6.7]{cdmeetscat} it is enough to prove that for $\m_{\sX}$-a.e. point $x\in X$ a blow up tangent space exists that is isometric to $\R^d$ for some $d\in \N$.

To see this we can follow the same strategy as in the proof of \cite[Theorem 6.2]{cdmeetscat}  based on a theorem of Preiss that says for $\m_{\sX}$-a.e. point $x\in X$ an iterated tangent space at $x$ shows up as tangent space a $x$, and  an isometric splitting theorem for tangent spaces of $X$. 
The splitting theorem follows by the almost splitting theorem by Tian and Zhang in \cite{tianzhang}. For $x\in X$ and a tangent space $T_xX$ at $X$  there exists a blow up sequence $(\epsilon_i M_i, x_i)$ that converges in measured Gromov-Hausdorff sense to $T_xX$. 
If $T_xX$ contains a geodesic, in a small ball around $x_i$ we can find points $p_i^+, p_i^-$ such that 
\begin{align*}
d_{M_i}(p^+,x_i)+d_{M_i}(p^-_i,x_i)- d_{M_i}(p^+_i,p^-_i)\leq \delta_i d_{M_i}(p^+_i,p^-_i).
\end{align*}
for $\delta_i\rightarrow 0$. By the almost splitting theorem \cite[Theorem 2.31]{tianzhang} given $\epsilon>0$ we can choose $\delta_i>0$ sufficiently small such that $d_{GH}(B_\delta(x_i), B_\delta(0))\leq \epsilon$ where $B_{\delta}(0)\subset \R\times Y$ for some metric space $Y$ with $\diam Y\leq \pi$.
Hence, after rescaling appropriately and taking the limit for $\epsilon\rightarrow 0$, we see that $T_xX$ splits off $\R$ isometrically. 

By Preiss's theorem for $\m_{\sX}$-a.e. we can iterate this process and obtain some tangent cone $Z_x$ at $x$ that is isometric to $\R^k$ for some $k\leq n$.
This finishes the proof.
\end{proof}

\begin{corollary}
Let $X$ be a measure Gromov-Hausdorff limit of sequence of Riemannian manifolds satisfying \eqref{uniform2} and \eqref{noncollapsing}. Then every tangent space $T_xX$ for $x\in X$ satisfies the condition $RCD(0,N)$. In particular, $T_xX$ is an Euclidean cone over some $RCD(N-2,N-1)$ space $Y$.
\end{corollary}
\begin{proof} 
By \cite[Theorem 2.33]{tianzhang} (blow up) tangent cones are always Euclidean cones.
For a  blow up sequence $(\epsilon_i M_i, o_i)$ we can compute that 
\begin{align*}
k_{[\epsilon_i M_i, o_i]}({\epsilon_i}\kappa_i, p,\epsilon_i K, R)&= k_{[M_i,x_i]}(\kappa_i, p, K, \epsilon_i R)\\
&= \frac{(\epsilon_i R)^{2p}}{\epsilon^n_i} \int (\kappa_i-K)_{-}^p d\m_{M_i}\rightarrow 0.
\end{align*}
Hence, by the previous theorem $T_oX$ satisfies $RCD(0,N)$. The last part follows from \cite{ketterer2}.
\end{proof}
\noindent
\small{
\bibliography{new}

\bibliographystyle{amsalpha}}
\end{document}